\documentclass[11pt]{article}

\usepackage{ulem}
\usepackage{framed}
\usepackage{multienum}
\usepackage{multicol}
\usepackage{amsthm}
\usepackage{amssymb}
\usepackage{amscd}
\usepackage{amsmath}
\usepackage[all]{xy}
\usepackage[top=30truemm,bottom=30truemm,left=25truemm,right=25truemm]{geometry}
\usepackage{comment}
\usepackage{algorithmic,algorithm}
\usepackage{here}
\usepackage{graphicx}
\usepackage{slashbox}
\usepackage{color} %
\usepackage{enumerate} %
\usepackage{listliketab}
\usepackage{multirow}
\usepackage{setspace}


\makeatletter
  
  \@addtoreset{algorithm}{section}
\makeatother

\theoremstyle{definition}

\theoremstyle{definition}

\theoremstyle{plain}

\theoremstyle{plain}

\theoremstyle{plain}

\theoremstyle{plain}

\theoremstyle{plain}
\newtheorem{thm}{Theorem}[subsection]

\theoremstyle{definition}

\theoremstyle{definition}

\theoremstyle{definition}

\theoremstyle{definition}
\newtheorem{defin}[thm]{Definition}

\theoremstyle{definition}
\newtheorem{rem}[thm]{Remark}

\theoremstyle{plain}
\newtheorem{prop}[thm]{Proposition}

\theoremstyle{plain}
\newtheorem{lem}[thm]{Lemma}

\theoremstyle{plain}
\newtheorem{cor}[thm]{Corollary}

\theoremstyle{definition}

\theoremstyle{definition}

\theoremstyle{definition}
\newtheorem{conje}[thm]{Conjecture}

\theoremstyle{definition}

\theoremstyle{definition}

\numberwithin{equation}{subsection}

\def\Z{\mathbb{Z}}

\def\F{\mathbb{F}}

\newcommand{\GL}{\operatorname{GL}}

\newcommand{\gA}{\operatorname{A}}
\newcommand{\gB}{\operatorname{B}}
\newcommand{\gC}{\operatorname{C}}
\newcommand{\gD}{\operatorname{D}}

\newcommand{\gH}{\operatorname{H}}

\newcommand{\gO}{\operatorname{O}}

\newcommand{\gT}{\operatorname{T}}
\newcommand{\gU}{\operatorname{U}}
\newcommand{\gV}{\operatorname{V}}
\newcommand{\gW}{\operatorname{W}}

\newcommand{\Aut}{\operatorname{Aut}}

\newcommand{\diag}{\operatorname{diag}}

\makeatletter
\def\@seccntformat#1{\csname the#1\endcsname. }
\makeatother
\makeatletter
\renewcommand\section{\@startsection {section}{1}{\z@}%
 {-3.5ex \@plus -1ex \@minus -.2ex}%
 {2.3ex \@plus.2ex}%
 {\normalfont\large\bfseries}}
\makeatother

\begin{document}

\title
{\bf Automorphism groups of superspecial curves \\of genus $4$ over $\F_{11}$}
\author
{Momonari Kudo\thanks{Institute of Mathematics for Industry, Kyushu University.
E-mail: \texttt{m-kudo@math.kyushu-u.ac.jp}},
\ Shushi Harashita\thanks{Graduate School of Environment and Information Sciences, Yokohama National University.
E-mail: \texttt{harasita@ynu.ac.jp}}
\ and Hayato Senda\thanks{Graduate School of Environment and Information Sciences, Yokohama National University.
E-mail: \texttt{senda-hayato-fw@ynu.jp}}}
\providecommand{\keywords}[1]{\textit{Key words:} #1}
\maketitle

\begin{abstract}
In this paper, we explicitly determine the automorphism group of
every nonhyperelliptic superspecial curve of genus $4$ over $\mathbb{F}_{11}$.
Our algorithm determining automorphism groups works for any nonhyperelliptic curves of genus $4$ over finite fields.
With this computation, we show the compatibility between the enumeration
of superspecial curves of genus $4$ over $\mathbb{F}_{11}$
obtained computationally by the first and second authors in 2017
and an enumeration by Galois cohomology theory.
\end{abstract}


\section{Introduction}
For a nonsingular algebraic curve $C$
over a field $K$ of positive characteristic,
we call $C$ {\it superspecial} if its Jacobian is isomorphic to a product of supersingular elliptic curves over
the algebraic closure $\overline{K}$.
It is known that any superspecial curve over a prime field $\F_p$
has an $\F_{p^2}$-form which is a maximal curve over $\F_{p^2}$.
Here for a variety $X$ over a field $K$, 
a variety $Y$ over a field $E$ with $K \subset E \subset \overline K$
is called an {\it $E$-form of $X$}
if $X_{\overline K} \simeq Y_{\overline K}$ (cf. \cite[Chap. III, \S 1]{S}).
Automorphism groups of algebraic curves have been studied
as they are fundamental objects in the theory of algebraic geometry,
and also from the practical point of view
for instance they are necessary to be studied when we construct a secure cryptography from the Jacobian of an algebraic curve.

In \cite{KH17} and \cite{KH17a}, the first and second authors enumerated
nonhyperelliptic superspecial curves of genus $4$
over prime fields of characteristic $p\le 11$.
In this paper, restricting ourselves to the case of $p=11$,
we determine the explicit structure
of the automorphism group of each of them,
where automorphism groups are considered over
$\F_{11}$ and over its algebraic closure $\overline{\F_{11}}$.
The quadratic forms defining
our curves over $\F_{11}$ are divided into three types {\bf (N1)}, {\bf (N2)} and {\bf (Dege)}, see Subsection \ref{subsec:ortho}.
Our main results on the structures of the automorphism groups
over $\F_{11}$ are stated separately for these types:
Theorems \ref{theo:N1Aut}, \ref{theo:N2Aut} and  \ref{theo:DegenerateAut}.
The structures of the automorphism groups over $\overline{\F_{11}}$
are determined in Theorem \ref{theo:AlcAut}. Moreover
we give explicit generators of each automorphism group in
Propositions \ref{prop:N1Aut}, \ref{prop:N2Aut}, \ref{prop:DegeAut} and \ref{prop:ACAut} respectively.
Finally, from the results over $\overline{\F_{11}}$
(Theorem \ref{theo:AlcAut} and Proposition \ref{prop:ACAut}),
we show the compatibility between the enumeration in
\cite{KH17} and \cite{KH17a}
and an enumeration by Galois cohomology theory,
which strongly supports the truth of 
results in \cite{KH17} and \cite{KH17a}
obtained by the computational enumeration,
i.e., that there is no (crucial) bug
in the programs used in our computational proof.

For superspecial curves, here is an extra reason to study automorphism groups.
A fundamental question asks us
whether superspecial curves of a given type
(genus, characteristic, hyperelliptic or not, and so on)
exist or not, and if they exist, then enumerate them.
Determining the automorphism groups contributes to solving these problems.
Recall the mass formula,
which counts principally polarized superspecial abelian varieties with weight by
the orders of automorphisms of them:
\begin{equation}\label{MassFormula}
M_g(p) := \sum_{(A,\Theta)}\frac{1}{\sharp \Aut(A,\Theta)} = \prod_{i=1}^g\frac{(2i-1)!\zeta(2i)}{(2\pi)^{2i}}\prod_{i=1}^g(p^i+(-1)^i),
\end{equation}
where $(A,\Theta)$ runs through isomorphism classes of principally polarized abelian varieties. It is much easier to get the mass than to get the exact number
of the isomorphism classes of principally polarized superspecial abelian varieties.

Let $J(C)$ denote the Jacobian variety with principal polarization
associated to a nonsingular curve $C$.
It is known that $\Aut(C)\simeq \Aut(J(C))$
if $C$ is hyperelliptic and $\Aut(C)\times\{\pm 1\}\simeq \Aut(J(C))$ otherwise.
As the polarizations of $J(C)$ for nonsingular curves $C$ are indecomposable,
we are interested in the indecomposable part
$M'_g(p)$ of $M_g(p)$, which is the sum  \eqref{MassFormula} over 
superspecial abelian varieties with indecomposable principal polarization.
The value $M'_4(p)$ for $g=4$ is computed as
\[
M_4'(p) = M_4(p) - M_3(p)M_1(p) - \frac{1}{2} M_2(p)^2 + M_2(p)M_1(p)^2-\frac{1}{4}M_1(p)^4
\]
by removing the contribution of decomposable ones from $M_4(p)$.
This is done by the same way as in \cite[Lemma 2.2]{Ibukiyama},
but we omit the detail as this is not the main theme of this paper.

The mass from our curves ($C_1,\ldots,C_9$ in Theorem \ref{theo:AlcAut}) is $5/8$.
This amount looks small, compared
with the whole mass $M_4(11)=8485039/497664$ and
also with $M'_4(11) = 1395421/82944$.
This means that there exist
superspecial curves of other types
(hyperelliptic or not descendible to $\F_{11}$)
or (principally polarized) superspecial abelian varieties
which are not obtained as Jacobians.
As a future work,
it would be interesting to know how much of $M_4(11)$ each of them occupies,
in particular the case of $g=4$ and $p=11$ of more general conjecture:
\begin{conje}
Let $p$ be an arbitrary prime number.
If $g\ge 4$, then there would exist
a superspecial abelian variety of dimension $g$
in characteristic $p$
with indecomposable principal polarization
which is not the Jacobian variety of any nonsingular curve.
\end{conje}

We propose this conjecture,
as this is true for $g=4$ and $p\le 7$ (cf. \cite{KH16}),
comparing $M'_g(p)$
and the sum $M_g^{\text{curve}}(p)$ of $1/|\Aut(J(C))|$ for superspecial nonsingular curves $C$ of genus $g$ over $\overline{\F_p}$, and is more likely to hold for larger $g$ and $p$. 
\begin{table}[htb]
\begin{center}
  \begin{tabular}{|l||c|c|c|c|c|} \hline
$p$ & $2$ & $3$ & $5$ & $7$ & $11$ \\ \hline\hline
$M'_4(p)$ & $\frac{1}{3317760}$ & $\frac{1}{46080}$ & $\frac{539}{103680}$ & $\frac{173}{1024}$ & $\frac{1395421}{82944}$\rule[-2.4mm]{0mm}{7mm}\\ \hline
$M_4^{\text{curve}}(p)$  & $0$ & $0$ & $\frac{1}{720}$ & $0$ & $\ge \frac{5}{8}$\rule[-2.4mm]{0mm}{7mm} \\ \hline
  \end{tabular}
\end{center}
\end{table}

Let us give an overview of this paper.
In Section 2, we review some facts shown in \cite{KH16} and \cite{KH17}
on superspecial curves of genus 4
and on their automorphisms,
and present an algorithm to determine the points of the variety
defined by a zero-dimensional ideal.
In Section 3 we describe an algorithm determining
the automorphism group of a given nonhyperelliptic curve of genus 4.
In Section 4, our main results are stated with their proofs.
In Section 5, we show that the enumeration
in \cite{KH17} and \cite{KH17a} with our computation of automorphisms
is compatible with that by Galois cohomology theory 
together with our results over an algebraically closed field.

\subsection*{Acknowledgments}
The authors thank Professor Tomoyoshi Ibukiyama for his valuable comments.
This work was supported by
JSPS Grant-in-Aid for Scientific Research (C) 17K05196.
\section{Preliminaries}\label{section2}

In this section, we give a brief review of previous results on the enumeration of (nonhyperelliptic) superspecial curves of genus $4$ over $\F_{11}$.
Moreover we shall introduce computational techniques for solving a system of multivariate polynomial equations via the Gr\"{o}bner basis computation.
The computational techniques described in this section will be used to prove our main results in this paper.

\subsection{Superspecial curves}\label{subsec:ssp}

Let $K$ be a perfect field with $\mathrm{char}(K)=p>0$, and let $\overline{K}$ denote its algebraic closure.
A curve $C$ of genus $g$ over $K$ is said to be {\it superspecial} if its Jacobian $J (C)$ is isomorphic to $E^g$ over $\overline{K}$ for a supersingular elliptic curve $E$.
Here is a classical problem: Given $g$ and $K$, determine the (non-)existence of a superspecial curve of genus $g$ over $K$.
Moreover if such a curve exists, enumerate all $K$-isomorphism classes and all $\overline{K}$-isomorphism classes of superspecial curves of genus $g$ over $K$.

In the case of $g=4$ and $K = \mathbb{F}_{11}$, the $K$-isomorphism classes and the $\overline{K}$-isomorphism classes of all nonhyperelliptic superspecial curves of genus $4$ over $K$ are determined as follows:

\begin{thm}[\cite{KH17}, Theorem B]\label{thm:KH17overF11}
There exist precisely $30$ nonhyperelliptic superspecial curves of genus $4$ over $\mathbb{F}_{11}$ up to isomorphism over $\mathbb{F}_{11}$. 
The thirty isomorphism classes are given by
\begin{description}
	\item[{\bf (N1)}] $V(Q^{\rm (N1)}, P_i^{{\rm (N1)}})$ with $Q^{\rm (N1)}=2 x w + 2 y z$ for $1 \leq i \leq 8$,
	\item[{\bf (N2)}] $V(Q^{\rm (N2)}, P_j^{\rm (N2)})$ with $Q^{\rm (N2)}=2 x w + y^2 - \epsilon z^2$ for $\epsilon \in \mathbb{F}_{11}^{\times} \smallsetminus (\mathbb{F}_{11}^{\times})^2$ and for $1 \leq j \leq 5$, and
	\item[{\bf (Dege)}] $V(Q^{\rm (Dege)},P_k^{{\rm (Dege)}})$ with $Q^{\rm (Dege)}= 2 y w + z^2$ for $1 \leq k \leq 17$.
\end{description}
Here each of $P_i^{\rm (N1)}$, $P_j^{\rm (N2)}$ and $P_k^{\rm (Dege)}$ is a cubic form in $\mathbb{F}_{11}[x,y,z,w]$, and is given at the beginning of {\rm Section \ref{section4}}.
\end{thm}

\begin{cor}\label{cor:KH17overF11}
Any nonhyperelliptic superspecial curve of genus $4$ over $\mathbb{F}_{11}$ is isomorphic over $\overline{\mathbb{F}_{11}}$ to one of the curves $V(Q^{{\rm (N1)}}, P_i^{({\rm alc})})$ for $1 \leq i \leq 3$, or $V(Q^{({\rm Dege})}, P_j^{({\rm alc})})$ for $4 \leq j \leq 9$, where $Q^{{\rm (N1)}} := 2 x w + 2 y z$, $Q^{({\rm Dege})}:= 2 yw + z^2$ and
\begin{eqnarray}
P_1^{(\mathrm{alc})}&:=& x^2 y + x^2 z + 2 y^2 z + 5 y^2 w + 9 y z^2 + y z w + 4 z^3 + 3 z^2 w + 10 z w^2 + w^3, \nonumber \\
P_2^{(\mathrm{alc})} &:=& x^2 y + x^2 z + y^3 + y^2 z + 7 y z^2 + 4 y w^2 + 2 z^3 + 9 z w^2, \nonumber \\
P_3^{(\mathrm{alc})} &:=& x^2 y + x^2 z + y^3 + 8 y^2 z + 3 y z^2 + 10 y w^2 + 10 z^3 + 10 z w^2, \nonumber \\
P_4^{(\mathrm{alc})} &:=& x^3 + y^3 + w^3, \nonumber \\
P_5^{(\mathrm{alc})} &:=& x^3 + y^3 + z^3 + 5 w^3, \nonumber \\
P_6^{(\mathrm{alc})} &:=& x^3 + x w^2 + y^3, \nonumber \\
P_7^{(\mathrm{alc})} &:=& x^3 + x z w + y^3 + 7 z^3 + w^3, \nonumber \\
P_8^{(\mathrm{alc})} &:=& x^3 + x y z + x w^2 + y^3 + 5 z^3 + 4 w^3, \nonumber \\
P_9^{(\mathrm{alc})} &:=& x^3 + x y z + 6 x w^2 + y^3 + 8 z^3 + 8 w^3. \nonumber 
\end{eqnarray}
\end{cor}

\subsection{Orthogonal groups}\label{subsec:ortho}

Let $p$ be an odd prime number, and $K$ a (perfect) field of characteristic $p$.
Let $Q$ be a quadratic form over $K$ in $x$, $y$, $z$ and $w$.
We denote by $\varphi$ the coefficient matrix of $Q$.
The orthogonal group of the quadratic form $Q$ is defined as follows:
\[ 
\gO_{\varphi}(K) := \{\, g \in \GL_{n}(K) \mid {}^{t}g \varphi g = \varphi \,\} , 
\]
where $\mathrm{GL}_n (K)$ denotes the general linear group of degree $n$ over $K$.
In addition, we define the orthogonal similitude group of the quadratic form $Q$ as follows:
\[ 
\tilde{\gO}_{\varphi}(K) := \{\, g \in \GL_{n}(K) \mid {}^{t}g \varphi g = \mu\varphi \mbox{ for some } \mu \in K^{\times} \,\}. 
\]
The quadratic forms defining nonhyperelliptic curves of genus $4$ over $K$ are classified into the following three types: {\bf (N1)} $Q^{({\rm N1})}=2 x w + 2 y z$, {\bf (N2)} $Q^{({\rm N2})} = 2 x w + y^2 - \epsilon z^2$ with $\epsilon \in K^{\times} \smallsetminus (K^{\times})^2$, and {\bf (Dege)} $Q^{({\rm Dege})} = 2 y w + z^2$ (cf. \cite[Remark 2.1.1]{KH16}). 
To compute the automorphism group of a (nonhyperelliptic) curve of genus $4$, we will give the Bruhat decomposition of the orthogonal group of each quadratic form.
Using the decomposition, computing automorphisms of a curve is reduced into solving a system of multivariate algebraic equations.

\subsubsection{The orthogonal group in (N1) case}\label{subsec:orthoN1}

The quadratic form $Q^{({\rm N1})}=2 x w + 2 y z$ is non-degenerate and its coefficient matrix is
\[
\begin{pmatrix}
0&0&0&1\\
0&0&1&0\\
0&1&0&0\\
1&0&0&0\\
\end{pmatrix}.
\]
By \cite[Subsection 4.1]{KH16}, we have the following decompositions:
\[
\gO_{\varphi}(K)=\gB \gW \gU \quad \text{and} \quad \tilde\gO_{\varphi}(K)=\tilde \gB \gW \gU
\]
with $\gB = \gA \gT \gU$ and $\tilde \gB = \gA \tilde \gT \gU$, where $\gA$, $\gT$, $\tilde\gT$, $\gW$ and $\gU$ are given as follows.
We set $\gT = \{ \diag(a,b,b^{-1},a^{-1}) \mid a,b \in K^\times\}$ , $\tilde \gT = \{\diag(a,b,cb^{-1},ca^{-1}) \mid a, b, c \in K^\times \}$,
\[
\gU = \left\{ \left. 
\begin{pmatrix}
1&a&0&0\\
0&1&0&0\\
0&0&1&-a\\
0&0&0&1
\end{pmatrix}
\begin{pmatrix}
1&0&b&0\\
0&1&0&-b\\
0&0&1&0\\
0&0&0&1
\end{pmatrix}
\right| a,b \in K \right\},\quad \gA = \left\{1_4, 
\begin{pmatrix}
1&0&0&0\\
0&0&1&0\\
0&1&0&0\\
0&0&0&1
\end{pmatrix}\right\}
\]
and $\gW=\{1_4, s_1, s_2, s_1 s_2 \}$ with
\[
s_1 = \begin{pmatrix}
0&1&0&0\\
1&0&0&0\\
0&0&0&1\\
0&0&1&0
\end{pmatrix}, \ \mbox{ and }
s_2 = \begin{pmatrix}
0&0&1&0\\
0&0&0&1\\
1&0&0&0\\
0&1&0&0
\end{pmatrix}.
\]

\subsubsection{The orthogonal group in (N2) case}\label{subsec:orthoN2}

The quadratic form $Q^{({\rm N2})} = 2 x w + y^2 - \epsilon z^2$ with $\epsilon \in K^{\times} \smallsetminus (K^{\times})^2$ is non-degenerate and its coefficient matrix is
\[
\begin{pmatrix}
0&0&0&1\\
0&1&0&0\\
0&0&-\epsilon &0\\
1&0&0&0
\end{pmatrix}. 
\]
By \cite[Subsection 4.1]{KH16}, we have the following decompositions:
\[
\gO_{\varphi}(K) = \gB \gW \gU \quad \text{and} \quad \tilde\gO_{\varphi}(K) = \tilde \gB \gW \gU
\]
with $\gB = \gA \gT \gU$ and $\tilde \gB = \gA \tilde \gT \gU$, where $\gA$, $\gT$, $\tilde\gT$, $\gW$ and $\gU$ are given as follows.
We set
$\gA =\{1_4, \diag(1,1,-1,1) \}$,
\[
\gU = \left\{\left.
\begin{pmatrix}
1 & a & 0 & -a^2/2\\
0 & 1 & 0 & -a\\
0 & 0 & 1 & 0\\
0 & 0 & 0 & 1
\end{pmatrix}
\begin{pmatrix}
1 & 0 & b & b^2/(2\epsilon )\\
0 & 1 & 0 & 0\\
0 & 0 & 1 & b/ \epsilon \\
0 & 0 & 0 & 1
\end{pmatrix}
\right| a,b \in K
\right\},
\]
\[
\gW = \left\{ 1_4, 
\begin{pmatrix}
0 & 0 & 0 & 1\\
0 & 1 & 0 & 0\\
0 & 0 & -1 & 0\\
1 & 0 & 0 & 0
\end{pmatrix}\right\},\quad
\tilde \gC=\left\{\left.
R(a,b):=
\begin{pmatrix}
1 & 0 & 0 & 0\\
0 & a & \epsilon b & 0\\
0 & b & a & 0\\
0 & 0 & 0 & a^2-\epsilon b^2
\end{pmatrix} \right|
\begin{matrix}
a,b \in K,\\
a^2 -\epsilon b^2 \ne 0
\end{matrix}
\right\}.
\]
Put $\gC = \{ R(a,b) \in \tilde \gC \mid a^2 - \epsilon b^2 = 1 \}$, $\gT = \gH \gC$ and $\tilde \gT= \gH\tilde \gC$, where $\gH = \{ \diag(a,1,1,a^{-1}) \mid a \in K^\times \}$.

\subsubsection{The orthogonal group in (Dege) case}\label{subsec:orthoDege}
The quadratic form $Q^{({\rm Dege})} = 2 y w + z^2$ is degenerate and its coefficient matrix is
\[
\begin{pmatrix}
0&0&0&0\\
0&0&0&1\\
0&0&1 &0 \\
0&1&0&0
\end{pmatrix}. 
\]
By \cite[Lemma 4.2.1]{KH16}, we have the following decompositions:
\[
\gO_{\varphi}(K) = (\gB \sqcup \gB M_{\rm W} \gU) \gV \quad \text{and} \quad 
\tilde \gO_{\varphi}(K) = (\tilde \gB \sqcup \tilde \gB M_{\rm W} \gU) \gV
\]
with $\gB = \gA \gT \gU$ and $\tilde \gB = \gA \tilde \gT \gU$, where $\gA = \{1_4, \diag(1,1,-1,1) \}$,
\[
\gT = \left\{\left. T(a):=
\begin{pmatrix}
1&0&0&0\\
0&a&0&0\\
0&0& 1&0\\
0&0&0&a^{-1}
\end{pmatrix} \right| a \in K^\times \right\},\ 
\gU = \left\{\left. U(a):=
\begin{pmatrix}
1&0&0&0\\
0&1&a&-a^2/2\\
0&0&1&-a\\
0&0&0&1
\end{pmatrix} \right| a \in K \right\},
\]
\[
M_{\rm W} =
\begin{pmatrix}
1&0&0&0\\
0&0&0&1\\
0&0&1&0\\
0&1&0&0
\end{pmatrix},\ 
\gV = \left\{\left.
\begin{pmatrix}
a&b&c&d\\
0&1&0&0\\
0&0&1&0\\
0&0&0&1
\end{pmatrix} \right| a \in K^\times \text{ and } b, c, d \in K \right\}
\]
and $\tilde \gT = \{ \diag(1,b,b,b) \mid b \in K^\times \} \gT$.

\subsection{Solving a system of polynomial equations of zero dimension}\label{subsec:solving}

In this subsection, we introduce methods for computing the points of the variety defined by a zero-dimensional ideal.
The computational methods given in this subsection shall be used to prove our main results in Section \ref{section4}.

Let $K$ be a field, and let $S:= K [X_1, \ldots , X_n]$ denote the polynomial ring of $n$ variables over $K$.
Put $S_{\overline{K}} := \overline{K} [ X_1, \ldots , X_n ]$.
For a subset $T \subset S$ and an extension field $L$ of $K$, we denote by $V_{L}(T)$ the zero-locus of $T$ over $L$, that is,
\[
V_{L} (T) = \{ (a_1, \ldots , a_n) \in L^{n} \mid f (a_1, \ldots , a_n) = 0 \mbox{ for all } f \in T \} .
\]
For an ideal $I \subset S$, put $R = S / I$, $I_{\overline{K}}:= \langle I \rangle_{S_{\overline{K}}}$, and $R_{\overline{K}}:=  S_{\overline{K}} / I_{\overline{K}} $.
Let $G$ be a Gr\"{o}bner basis of $I$ with respect to a term order $\succ$ on $X_1, \ldots , X_n$.
We set
\[
L (I, G) = \{ X \in \mathcal{M}_S \mid \mathrm{LM}_{\succ} (g) \nmid X \mbox{ for all } g \in G \} ,
\]
where $\mathcal{M}_S$ denotes the set of monomials in $S$, and $ \mathrm{LM}_{\succ} (g)$ denotes the leading monomial of $g$ with respect to $\succ$.
By the definition of the normal form, the set $L (I, G)$ gives rise to bases of the $K$-vector space $R=K[X_1, \ldots , X_n]/I$ and the $\overline{K}$-vector space $R_{\overline{K}}=\overline{K}[X_1, \ldots , X_n]/  I_{\overline{K}}$ with $I_{\overline{K}} = \langle I \rangle_{S_{\overline{K}}}$.

We first define a zero-dimensional ideal.

\begin{defin}[Zero-dimensional ideals]
An ideal $I \subset S=K [ X_1, \ldots , X_n ]$ is said to be {\it zero-dimensional} if $\# V_{\overline{K}} ( I) < \infty$.
\end{defin}

\begin{lem}\label{lem:elim2}
Let $I$ be an ideal in $S=K [ X_1, \ldots , X_n ]$, and $\succ$ a term order on $X_1, \ldots , X_n$.
Let $G$ be a Gr\"{o}bner basis for $I$ with respect to $\succ$.
Put $S_{\overline{K}} := \overline{K} [ X_1, \ldots , X_n ]$, $I_{\overline{K}}:= \langle I \rangle_{S_{\overline{K}}}$, and $R_{\overline{K}}:=  S_{\overline{K}} / I_{\overline{K}} $.
The following are equivalent:
\begin{enumerate}
\item[$(1)$] $\mathrm{dim}_{\overline{K}} R_{\overline{K}} < \infty$.
\item[$(2)$] For each $1 \leq i \leq n$, there exists $g_i \in G$ such that $\mathrm{LM}_{\succ} (g_i) = X_i^{k_i}$ for some $k_i \geq 0$.
\end{enumerate}
In particular, we have $I \cap K [ X_{i+1}, \ldots , X_{n} ] \neq \{ 0 \}$ for $0 \leq i \leq n-1$ by taking $\succ$ to be the lexicographical $($lex$)$ order with $X_1 \succ \cdots \succ X_n$.
\end{lem}

\begin{proof}
(1) $\Rightarrow$ (2): Assume (1).
We assume for a contradiction that for some $1 \leq i \leq n$, there does not exist $g \in G$ such that $\mathrm{LM}_{\succ} (g)$ is not divided by $X_{i^{\prime}}$ for any $i^{\prime} \neq i$.
In this case, the set
\[
L (I, G) = \{ X \in \mathcal{M}_S \mid \mathrm{LM}_{\succ} (g) \nmid X \mbox{ for all } g \in G \}
\]
includes the set $\{ X_i^{k} \mid k \geq 0 \}$ as a subset, where $\mathcal{M}_S$ denotes the set of monomials in $S$.
Since $L ( I, G )$ is a basis of the $\overline{K}$-vector space $R_{\overline{K}}$, we have $\mathrm{dim}_{\overline{K}} R_{\overline{K}} = \infty$.

(2) $\Rightarrow$ (1):
Assume (2).
We have $L (I, G) \subset \{ X_1^{i_1} \cdots X_n^{i_n} \mid i_j < k_j \mbox{ for all } 1 \leq j \leq n \}$, and hence $\mathrm{dim}_{\overline{K}} R_{\overline{K}} < \infty$.
\end{proof}

By using Hilbert's Nullstellensatz and Lemma \ref{lem:elim2}, we can show that $\# V_{\overline{K}} ( I ) < \infty$ if and only if $\mathrm{dim}_{\overline{K}} R_{\overline{K}} < \infty$.

\begin{lem}\label{lem:extension}
Let $I$ be a zero-dimensional ideal in $S=K [ X_1, \ldots , X_n ]$.
If there exists an element $(a_{i+1}, \ldots , a_{n} ) \in V_{\overline{K}} (I_{i}) $ with $I_i :=I \cap K [ X_{i+1}, \ldots , X_{n} ]$, then there exists $a_i \in \overline{K}$ such that $(a_i, \ldots , a_{n}) \in V_{\overline{K}} ( I_{i-1} )$ with $I_{i-1} :=I \cap K [ X_{i}, \ldots , X_{n} ]$.
\end{lem}

\begin{proof}
Assume $(a_{i+1}, \ldots , a_{n} ) \in V_{\overline{K}} (I_{i}) $ with $I_i :=I \cap K [ X_{i+1}, \ldots , X_{n} ]$.
Let $G$ be a Gr\"{o}bner basis for $I$ with respect to the lex order with $X_1 \succ \cdots \succ X_n$.
Since $\# V_{\overline{K}} ( I ) < \infty$, it follows from Lemma \ref{lem:elim2} that there exists $g_i \in G$ such that $g_i$ is of the form
\[
g_{i} = X_i^{k_i} + \sum_{j=0}^{k_i-1} h_{j}^{(i)} (X_{i+1}, \ldots , X_n ) X_i^j
\]
for some $k_i \geq 0$ and $h_j^{(i)} \in K [X_{i+1}, \ldots , X_n]$.
By the extension theorem, one has that there exists an element $a_i \in \overline{K}$ such that $(a_i, \ldots , a_{n}) \in V_{\overline{K}} ( I_{i-1} )$.
\end{proof}

\begin{lem}\label{lem:elim3}
Let $I$ be an ideal in $S=K [ X_1, \ldots , X_n ]$, and $\succ$ the lex order with $X_1 \succ \cdots \succ X_n$.
Let $G$ be a Gr\"{o}bner basis for $I$ with respect to $\succ$.
The following are equivalent:
\begin{enumerate}
\item[$(1)$] The ideal $I$ is zero-dimensional, i.e., $\# V_{\overline{K}} ( I ) < \infty$.
\item[$(2)$] There exists $g \in G$ such that $\mathrm{LM}_{\succ} (g) = X_1^{k_1}$ for some $k_1 \geq 0$, and $I_1 := I \cap K [ X_{2}, \ldots , X_{n} ]$ is a zero-dimensional ideal in $K[X_2, \ldots , X_n]$.
\end{enumerate}
\end{lem}

\begin{proof}
(1) $\Rightarrow$ (2): This is a special case of Lemma \ref{lem:elim2}.
We show $\# V_{\overline{K}}(I_1) < \infty$.
Let $(a_2, \ldots , a_n) \in V_{\overline{K}}(I_1)$.
Since $I$ is zero-dimensional, it follows from Lemma \ref{lem:extension} that there exists $a_1 \in \overline{K}$ such that $(a_1, \ldots , a_{n}) \in V_{\overline{K}} ( I )$.
Since $V_{\overline{K}}(I)$ is finite, $V_{\overline{K}}(I_1)$ is also finite.

(2) $\Rightarrow$ (1):
Assume (2).
By the elimination theorem, $G_1 := G \cap K[X_2, \ldots , X_n]$ is a Gr\"{o}bner basis for $I_1$ with respect to the order on $X_2 \succ \cdots \succ X_n$ induced by $\succ$.
Since $I_1$ is zero-dimensional, it follows from (1) $\Rightarrow$ (2) that there exists $g_2 \in G_1 \subset G$ such that $\mathrm{LM}_{\succ} (g_2)=X_2^{k_2}$ for some $k_2 \geq 0$, and $\# V_{\overline{K}} (I_2) < \infty $.
Here we set $I_2 = I_1 \cap K [X_3, \ldots , X_n]$.
Inductively for each $1 \leq i \leq n$, there exists $g_i \in G$ such that $\mathrm{LM}_{\succ} (g_i) = X_i^{k_i}$ for some $k_i \geq 0$ (and $I_{i-1} := I \cap K [ X_{i}, \ldots , X_{n} ]$ is a zero-dimensional ideal in $K[X_{i}, \ldots , X_n]$).
By Lemma \ref{lem:elim2}, the ideal $I$ is zero-dimensional.
\end{proof}

\begin{lem}\label{lem:extension2}
Let $I$ be an ideal in $S=K [ X_1, \ldots , X_n ]$, and let $I_1 := I \cap K [X_2, \ldots , X_n]$.
For an element $(a_1, \ldots , a_n) \in K^{n}$, the following are equivalent:
\begin{enumerate}
\item[$(1)$] $(a_1, \ldots , a_n) \in V_{\overline{K}}(I)$,
\item[$(2)$] $(a_2, \ldots , a_n) \in V_{\overline{K}}(I_1)$ and $a_1 \in V_{\overline{K}}( I (a_2, \ldots , a_n) )$, where $I (a_2, \ldots , a_n)$ denotes the ideal in $K [X_1]$ generated by $\{ f (X_1, a_2, \ldots , a_n ) : f \in I \}$.
\end{enumerate}
\end{lem}

Based on Lemmas \ref{lem:elim3} and \ref{lem:extension2}, we give an algorithm for computing the zero-locus of a given ideal.
Given an ideal $I$, the following algorithm computes the zero-locus $V_{\overline{K}} (I)$ if $I$ is zero-dimensional.

\begin{description}
\item[{\it Step 1.}] First we compute a Gr\"{o}bner basis $G = \{ g_1, \ldots , g_k \}$ for $I$ with respect to the lex order with $X_1 \succ \cdots \succ X_n$.
For each $0 \leq i \leq n-1$, we set $G_i := G \cap K [X_{i+1}, \ldots , X_n]$.

\item[{\it Step 2.}] To determine $V_{\overline{K}}(I_{n-1}) = V_{\overline{K}}(G_{n-1})$, we proceed with the following steps:
	\begin{description}
	\item[{\it Step 2-1.}] From $G_{n-1} = G \cap K [X_n]$, compute a polynomial $h_n ( X_n )$ such that $\langle h_n \rangle = I \cap K [X_n]$.
	Note that $I_{n-1}$ is zero-dimensional and $G_{n-1} \neq \emptyset$ since $I$ is zero-dimensional.
	Indeed, if $G_{n-1} = \emptyset$, then $I_{n-1}$ is not zero-dimensional, and so is $I_{n-2}$ by Lemma \ref{lem:elim3}.
	Recursively $I$ is not zero-dimensional.
	\item[{\it Step 2-2.}] Computing the splitting field $K^{\prime}$ of $h_n$ together with its roots in $K^{\prime}$, we determine $V_{\overline{K}}(I_{n-1})$.
	Replace $K$ by $K^{\prime}$.
	\end{description}

\item[{\it Step 3.}] To determine $V_{\overline{K}}(I_{n-2}) = V_{\overline{K}}(G_{n-2})$, we proceed with the following steps for each element $a_n \in V_{\overline{K}}(I_{n-1})$:
	\begin{description}
	\item[{\it Step 3-1.}] Compute the set $G_{n-2}(a_n):=\{ h ( X_{n-1}, a_n ) \mid h \in G_{n-2} \}$.
	If $G_{n-2}(a_n) = \{ 0 \}$, then $I$ is not zero-dimensional.
Indeed, if $h ( X_{n-1}, a_n ) = 0$ for all $h \in G_{n-2}$, one has $h ( a_{n-1}, a_n ) = 0$ for all $a_{n-1} \in \overline{K}$.
In this case, $V_{\overline{K}} (I_{n-2})$ has infinitely many elements, and hence so is $V_{\overline{K}} (I_{n-3})$ by Lemma \ref{lem:elim3}.
Recursively $V_{\overline{K}} (I)$ has infinitely many elements.
	\item[{\it Step 3-2.}] Computing the splitting field $K^{\prime}$ of the polynomial $\mathrm{GCD} \{ h^{\prime}(X_{n-1}) \mid h^{\prime} \in G_{n-2}(a_n) \}$ together with its roots in $K^{\prime}$, we determine $V_{\overline{K}}(I_{n-1}(a_n))$.
	Replace $K$ by $K^{\prime}$.
	\end{description}
	By Lemma \ref{lem:extension2}, we compute $V_{\overline{K}}(I_{n-2})$.
	
\item[{\it Step 4.}] Proceed as before until $V_{\overline{K}}(I)$ is computed.
\end{description}

In the following, we give a pseudocode for the above algorithm to compute all solutions to a system of algebraic equations defined by a zero-dimensional ideal.
\begin{algorithm}[H] %
\caption{$\texttt{ZeroDimensionalIdealVariety} ( f_1, \ldots , f_s )$}
\label{alg:VarietyAC1}
\begin{algorithmic}[1]
\REQUIRE{An ordered $s$-tuple $( f_1, \ldots , f_s ) \in S^s$ such that $I:= \langle f_1, \ldots , f_s \rangle_S$ satisfies $\# V_{\overline{K}} (I) < \infty$}
\ENSURE{The set $V_{\overline{K}} ( I ) = \{ ( a_1, \ldots , a_n ) \in \overline{K}^n \mid f ( a_1, \ldots , a_n ) = 0 \mbox{ for all } f \in I \}$, and the smallest extension field $K^{\prime}$ of $K$ such that $V_{\overline{K}} (I) = V_{K^{\prime}}(I)$}
\STATE Compute a Gr\"{o}bner basis $G$ for $I$ with respect to the lex order with $X_1 \succ \cdots \succ X_n$
\STATE $G_{n-1}$ $\leftarrow$ $G \cap K [ X_{n}]$
\STATE $h_n ( X_n )$ $\leftarrow$ the minimal polynomial of the ideal generated by $\{ h ( X_{n}) \mid h \in G_{n-1} \}$
\STATE Construct the minimal splitting field $K^{\prime}$ of $h_n ( X_n ) $, and replace $K$ by $K^{\prime}$
\STATE $V_{n-1}$ $\leftarrow$ $\{ a_n \mid a_{n} \in K \mbox{ and } h_n ( a_{n} ) = 0 \}$
\FOR{$i=n-2$ \textbf{down to} $0$}
	\STATE $G_i$ $\leftarrow$ $G \cap K [X_{i+1}, \ldots , X_{n}]$
	\STATE $V_i$ $\leftarrow$ $\emptyset$
	\FOR{$(a_{i+2}, \ldots , a_n ) \in V_{i+1}$}
		\STATE Compute $\{ h ( X_{i+1}, a_{i+2}, \ldots , a_n ) \mid h \in G_i \}$
		\STATE $h_{i+1}(X_{i+1})$ $\leftarrow$ the minimal polynomial of the ideal generated by the set of the polynomial $h ( X_{i+1}, a_{i+2}, \ldots , a_n )$ for $h \in G_i$
		\STATE Construct the minimal splitting field $K^{\prime}$ of $h_{i+1}(X_{i+1})$, and replace $K$ by $K^{\prime}$
		\STATE $V_{i}$ $\leftarrow$ $V_{i} \cup \{ ( a_{i+1}, a_{i+2}, \ldots , a_n ) \mid a_{i+1} \in K \mbox{ and } h_{i+1} ( a_{i+1} ) = 0 \}$
	\ENDFOR
\ENDFOR
\RETURN $V_0$ and $K$
\end{algorithmic}
\end{algorithm}

\section{Algorithm to compute automorphism groups}\label{section3}
In this section, we shall give an algorithm for computing the automorphism groups of (nonhyperelliptic) superspecial curves of genus $4$.
Specifically, we reduce computing the automorphism groups into solving multivariate systems over a finite field via the Bruhat decomposition given in Subsection \ref{subsec:ortho}.
Each automorphism computed by our algorithm is represented as an element of the general linear group of degree $4$ (in fact, an element of the orthogonal group associated to some quadratic form).
Moreover, we present a method to determine the finite group structure of a given subgroup of the general linear group of degree $4$.
Our method is heuristic, but useful to determine the finite group structure of computed automorphism groups.

\subsection{Computing the automorphisms}\label{subsec:compaut}

Throughout this section, let $K$ be a finite field $\mathbb{F}_q$, or its algebraic closure $\overline{\mathbb{F}_q}$.
Let $C=V ( Q, P )$ be a (superspecial) curve defined by an irreducible quadratic form $Q$ and an irreducible cubic form $P$ in $K[x,y,z,w]$.
Note that $Q$ is assumed to be one of the following three types: {\bf (N1)} $Q^{\rm (N1)} = 2 x w + 2 y z$, {\bf (N2)} $Q^{\rm (N2)} = 2 x w + y^2 - \epsilon z^2$ with $\epsilon \in K^{\times} \smallsetminus (K^{\times})^2$, and {\bf (Dege)} $Q^{\rm (Dege)} = 2 y w + z^2$.
We denote by $\Aut_K (C)$ the automorphism group of the curve $C$ over $K$, and denote it simply by $\Aut(C)$ if $K$ is algebraically closed.
Putting $G_K := \{\, g \in \tilde{\gO}_{\varphi}(K) \mid g \cdot P \equiv r P \bmod Q \mbox{ for some } r \in K^{\times} \,\}$, we have $\Aut_K(C)=G_K/{\sim}$, where $\varphi$ denotes the matrix associated to $Q$, and we write $g \sim cg$ for $g \in \tilde{\gO}_{\varphi}(K)$ and $c \in K^{\times}$.
Note that $g \cdot P := P ( (x, y, z, w) \cdot {}^t g )$ for $g \in \mathrm{GL}_{4} (K)$.
This subsection gives an algorithm to enumerate all representatives $g$ of the elements in $\mathrm{Aut}_{K}(C)$.

Here, we reduce computing all representatives of the elements in $\Aut_K(C)=G_K/{\sim}$ into solving multivariate systems over $K$ via the Bruhat decomposition given in Subsection \ref{subsec:ortho}.
To do this, we first regard some entries of elements in $\tilde{\mathrm{O}}_{\varphi}(K)$ as variables.
Specifically, using the Bruhat decomposition, we represent unknown elements $g$ of the group $G_K$ as a multiple of matrices with some unknown variables in their entries.
As we describe below, we have some patterns representing $g$ (e.g., in the case {\bf (N1)}, we have $8$ patterns $g_1, \ldots , g_8$).
For each element of each pattern, we construct a multivariate system from the equation
\[
g_i \cdot P \equiv r P \bmod Q \mbox{ with } r \neq 0 ,
\]
where $r$ is an extra variable.
Hence we obtain all elements $g$ in $G_K$ by solving the multivariate system for each $g_i$.
Moreover, for the computed set $G_K$, we can enumerate all representatives of the elements of $\Aut_K(C)=G_K /{\sim}$ by identifying $g$ with $c g$ for $g \in G_K$ and $c \in K^{\times}$.
Here, based on the above strategy, we give an algorithm for computing all elements $g$ in $G_K$.

\paragraph{Main algorithm}
Put $G_K = \emptyset$.
With notation as above, proceed with the following procedures for each $g_i \in \mathcal{G}_Q$ (we describe the set $\mathcal{G}_Q$ of matrices for each of {\bf (N1)}, {\bf (N2)} and {\bf (Dege)} below):
\begin{enumerate}
	\item[{\rm 1}.] Compute the zeros of a multivariate system derived from the equation
	\[
	g_i \cdot P \equiv r P \bmod Q ,
	\]
	where $r$ is an extra variable.
	Specifically, take the following three procedures:
	\begin{enumerate}
		\item[{\rm 1-1}.] Compute $P^{\prime}:=( g_i \cdot P -  r P ) \bmod Q$, where $g_i \cdot P := P ( (x, y, z, w) \cdot {}^t g_i )$.
		\item[{\rm 1-2}.] Let $\mathcal{P}_0$ be the set of the coefficients in $P^{\prime}$, and set $\mathcal{P}:= \mathcal{P}_0 \cup \mathcal{P}_Q$ (we describe the set $\mathcal{P}_Q$ of polynomials for each of {\bf (N1)}, {\bf (N2)} and {\bf (Dege)} below).
		\item[{\rm 1-3}.] Solve the multivariate system constructed in Step 1-2. Specifically solve the multivariate system $f=0$ for all $f \in \mathcal{P}$ over $K$ by Algorithm \ref{alg:VarietyAC1} in Subsection \ref{subsec:solving}.
	\end{enumerate}
	\item[{\rm 2}.] For each root, substitute it into $g_i$ and replace $G_K$ by $G_K \cup \{ g_i \}$.
\end{enumerate}

In Algorithm \ref{alg:AutoN1} below, we give a pseudocode of Main algorithm for $K=\mathbb{F}_q$ or $\overline{\mathbb{F}_q}$. 

\begin{algorithm}[H]
\caption{$\texttt{AutomorphismGroup}(Q,P,q,q^{\prime})$}
\label{alg:AutoN1}
\begin{algorithmic}[1]
\REQUIRE{A quadratic form $Q \in \mathbb{F}_q[x,y,z,w]$, a cubic form $P \in \mathbb{F}_q[x,y,z,w]$, the order $q$ of the finite field $\mathbb{F}_q$, and an integer $q^{\prime} \in \{ 0, q \}$}
\ENSURE{The set $G_K:= \{\, g \in \tilde{\gO}_{\varphi}(K) \mid g \cdot P \equiv r P \bmod{Q} \mbox{ for some } r \in K^{\times} \,\}$, where $K = \mathbb{F}_q$ (resp.\ $K = \overline{\mathbb{F}_q}$) if $q^{\prime}=q$ (resp.\ $q^{\prime}=0$)}
\STATE $G$ $\leftarrow$ $\emptyset$
\STATE $\varphi$ $\leftarrow$ the coefficient matrix of the quadratic form $Q$
\FOR{$g_i \in \mathcal{G}_Q$}
	\STATE $\mathcal{P}_0$ $\leftarrow$ $\emptyset$
	\STATE $(x^{\prime}, y^{\prime}, z^{\prime}, w^{\prime})$ $\leftarrow$ $(x,y,z,w) \cdot {}^t g_i$
	\STATE $P^{\prime}$ $\leftarrow$ $\left( P (x^{\prime}, y^{\prime}, z^{\prime}, w^{\prime}) - r P \right) \bmod Q$
	\STATE /* $r$ is a variable */
	\FOR{$i$, $j$, $k$, $\ell$ in $\mathbb{Z}_{\geq 0}$ with $i + j + k + \ell =3$}
		\STATE $\mathcal{P}_0$ $\leftarrow$ $\mathcal{P}_0 \cup \{ \mbox{the coefficient of } x^i y^j z^k w^{\ell} \mbox{ in } P^{\prime} \}$
	\ENDFOR
	\STATE $\mathcal{P}$ $\leftarrow$ $\mathcal{P}_0 \cup \mathcal{P}_Q$
	\STATE /* $\mathcal{P}_Q$ for each of {\bf (N1)}, {\bf (N2)} and {\bf (Dege)} is described below */
	\STATE Solve the multivariate system $f=0$ for all $f \in \mathcal{P}$ over $\overline{\mathbb{F}_{q}}$ (if $q^{\prime} =0$), or $\mathbb{F}_{q}$ (if $q^{\prime} =q$)
\STATE /* Use Algorithm \ref{alg:VarietyAC1} in Subsection \ref{subsec:solving} */
	\IF{$q^{\prime} = 0$} 
		\STATE $V$ $\leftarrow$ $V_{\overline{\mathbb{F}_q}}( \mathcal{P})$ 
	\ELSE
		\STATE $V$ $\leftarrow$ $V_{\mathbb{F}_q} ( \mathcal{P})$ 
	\ENDIF
	\FOR{$\underline{a}  \in V$}
		\STATE Substitute $\underline{a}$ into corresponding variables of unknown coefficients in $g_i$
		\STATE $G$ $\leftarrow$ $G \cup \{g_i \}$
	\ENDFOR
\ENDFOR
\RETURN $G$
\end{algorithmic}
\end{algorithm}

\paragraph{Elements in $\tilde\gO_{\varphi}(K)$ in (N1) case}
Let $Q = 2 x w + 2 y z$, and $\varphi$ its coefficient matrix.
Recall from Subsection \ref{subsec:orthoN1} that we have the Bruhat decomposition of the orthogonal similitude group associated to $\varphi$, say
\begin{eqnarray}
 \tilde\gO_{\varphi}(K)= \gA \tilde \gT \gU \gW \gU ,\label{eq:BruhatN1}
\end{eqnarray}
where $\gA$, $\tilde\gT$, $\gU$ and $\gW$ are the same as in Subsection \ref{subsec:orthoN1}.
Here writing $\gA = \{ 1_4, M_{\gA}:= E_{1,1} + E_{2,3} + E_{3,2} + E_{4,4} \}$ and $\gW = \{1_4, s_1, s_2, s_1 s_2 \}$, it follows from \eqref{eq:BruhatN1} that we have $\tilde\gO_{\varphi}(K) = \bigcup_{i=1}^8 \Omega_i$, where
\begin{equation*}
\begin{array}{llll}
\Omega_1 := \tilde{\gT} \gU \gU, & \Omega_2 := \tilde{\gT} \gU s_1 \gU, & \Omega_3 := \tilde{\gT} \gU s_2 \gU, & \Omega_4 := \tilde{\gT} \gU s_1 s_2 \gU, \\
\Omega_5 := M_{\gA} \tilde{\gT} \gU \gU, & \Omega_6 := M_{\gA} \tilde{\gT} \gU s_1 \gU, & \Omega_7 := M_{\gA} \tilde{\gT} \gU s_2 \gU, & \Omega_8 := M_{\gA} \tilde{\gT} \gU s_1 s_2 \gU .
\end{array}
\end{equation*}
Put $\tilde{T} (a_1,a_2, b_1,b_2,c) = \diag(a_1, b_1, c b_2, c a_2)$,
\[
U_1(d_1)=\begin{pmatrix}1&d_1&0&0\\0&1&0&0\\0&0&1&-d_1\\0&0&0&1\end{pmatrix}, \quad
U_2(d_2)=\begin{pmatrix}1&0&d_2&0\\0&1&0&-d_2\\0&0&1&0\\0&0&0&1\end{pmatrix}, \quad
U(d_1,d_2) = U_1(d_1) U_2(d_2),
\]
\[
U_1(e_1)=\begin{pmatrix}1&e_1&0&0\\0&1&0&0\\0&0&1&-e_1\\0&0&0&1\end{pmatrix}, \quad
U_2(e_2)=\begin{pmatrix}1&0&e_2&0\\0&1&0&-e_2\\0&0&1&0\\0&0&0&1\end{pmatrix}, \quad
U(e_1,e_2) = U_1(e_1) U_2(e_2),
\]
where $a_1$, $a_2$, $b_1$, $b_2$, $c$, $d_1$, $d_2$, $e_1$ and $e_2$ are elements in $K$ with $a_1 a_2 =1$, $b_1 b_2 =1$ and $c \neq 0$.
One can verify that elements of $\Omega_{i}$ for each $1 \leq i \leq 8$ are written as follows:
\begin{description}
\item[{\it Elements in $\Omega_1$.}]
$g_1 := \tilde{T} (a_1,a_2, b_1,b_2,c) U (d_1, d_2 )$,
\item[{\it Elements in $\Omega_2$.}]
$g_2 := \tilde{T} (a_1,a_2, b_1,b_2,c) U (d_1, d_2) s_1 U_1 (e_1)$,
\item[{\it Elements in $\Omega_3$.}]
$g_3 := \tilde{T} (a_1,a_2, b_1,b_2,c) U (d_1, d_2) s_2 U_2 (e_2)$,
\item[{\it Elements in $\Omega_4$.}]
$g_4 := \tilde{T} (a_1,a_2, b_1,b_2,c) U (d_1, d_2) s_1 s_2 U (e_1, e_2 )$,
\item[{\it Elements in $\Omega_5$.}]
$g_5 := M_{\gA} \tilde{T} (a_1,a_2, b_1,b_2,c) U (d_1, d_2 )$,
\item[{\it Elements in $\Omega_6$.}]
$g_6 := M_{\gA} \tilde{T} (a_1,a_2, b_1,b_2,c) U (d_1, d_2) s_1 U_1 (e_1)$,
\item[{\it Elements in $\Omega_7$.}]
$g_7 := M_{\gA} \tilde{T} (a_1,a_2, b_1,b_2,c) U (d_1, d_2) s_2 U_2 (e_2)$,
\item[{\it Elements in $\Omega_8$.}]
$g_8 := M_{\gA} \tilde{T} (a_1,a_2, b_1,b_2,c) U (d_1, d_2) s_1 s_2 U (e_1, e_2 )$.
\end{description}
Regarding $a_1$, $a_2$, $b_1$, $b_2$, $c$, $d_1$, $d_2$, $e_1$ and $e_2$ as variables, we set $\mathcal{G}_Q = \{ g_i \mid 1 \leq i \leq 8 \}$ in Main algorithm and Algorithm \ref{alg:AutoN1}. 
Moreover, we set $\mathcal{P}_Q = \{ a_1 a_2 - 1, b_1 b_2 - 1, c s - 1, r v -1 \}$, where $s$, $r$ and $v$ are extra variables.
If the $3$-rd power map is surjective over $K$, we may assume $\mathcal{P}_Q = \{ a_1 a_2 - 1, b_1 b_2 - 1, c s - 1, r -1 \}$.

\paragraph{Elements in $\tilde\gO_{\varphi}(K)$ in (N2) case}
Let $Q =2 x w + y^2 - \epsilon z^2$ with $\epsilon \in K^{\times} \smallsetminus (K^{\times})^2$, and $\varphi$ its coefficient matrix.
Recall from Subsection \ref{subsec:orthoN2} that we have the Bruhat decomposition of the orthogonal similitude group associated to $\varphi$, say
\begin{eqnarray}
 \tilde\gO_{\varphi}(K)= \gA \gH \tilde \gC \gU \gW \gU ,\label{eq:BruhatN2}
\end{eqnarray}
where $\gA$, $\gH$, $\tilde \gC$, $\gU$ and $\gW$ are the same as in Subsection \ref{subsec:orthoN2}.
Here writing $\gA = \{ 1_4, M_{\gA}:= \mathrm{diag}(1,1,-1,1) \}$ and $\gW =\{1_4,  M_{\gW}:= E_{1,4} + E_{2,2} - E_{3,3} + E_{4,1} \}$, it follows from \eqref{eq:BruhatN2} that we have $\tilde\gO_{\varphi}(K) = \bigcup_{i=1}^4 \Omega_i$, where
\begin{equation*}
\begin{array}{llll}
\Omega_1 := \gH \tilde \gC \gU \gU, & \Omega_2 := \gH \tilde \gC \gU M_{\gW} \gU, & \Omega_3 := M_{\gA} \gH \tilde \gC \gU \gU, & \Omega_4 := M_{\gA} \gH \tilde \gC \gU M_{\gW} \gU .
\end{array}
\end{equation*}
We set
\[
U (c_1,c_2) = 
\begin{pmatrix}
1&c_1&0&-c_1^2/2\\
0&1&0&-c_1\\
0&0&1&0\\
0&0&0&1
\end{pmatrix}
\begin{pmatrix}
1&0&c_2&c_2^2/(2 \epsilon)\\
0&1&0&0\\
0&0&1&c_2/ \epsilon \\
0&0&0&1
\end{pmatrix},
\]
\[
U (d_1, d_2) = 
\begin{pmatrix}
1&d_1&0&-d_1^2/2\\
0&1&0&-d_1\\
0&0&1&0\\
0&0&0&1
\end{pmatrix}
\begin{pmatrix}
1&0&d_2&d_2^2/(2 \epsilon)\\
0&1&0&0\\
0&0&1&d_2/ \epsilon \\
0&0&0&1
\end{pmatrix},
\]
\[
R(b_1,b_2) = \begin{pmatrix}
1 & 0 & 0 & 0\\
0 & b_1 & \epsilon b_2 & 0\\
0 & b_2 & b_1 & 0\\
0 & 0 & 0 & b_1^2-\epsilon b_2^2
\end{pmatrix},
\]
$H(a_1,a_2) = \diag(a_1,1,1, a_2)$ and $\tilde{T}(a_1,a_2,b_1,b_2) = H (a_1,a_2) R (b_1,b_2)$, where $a_1$, $a_2$, $b_1$, $b_2$, $c_1$, $c_2$, $d_1$ and $d_2$ are elements in $K$ with $a_1 a_2 =1$ and $b_1^2-\epsilon b_2^2 \neq 0$.
One can verify that elements of $\Omega_{i}$ for each $1 \leq i \leq 4$ are written as follows:
\begin{description}
\item[{\it Elements in $\Omega_1$.}]
$g_1 := \tilde{T}(a_1,a_2,b_1,b_2) U (c_1, c_2 )$,
\item[{\it Elements in $\Omega_2$.}]
$g_2 := \tilde{T}(a_1,a_2,b_1,b_2) U (c_1, c_2 ) M_{\gW} U (d_1, d_2 )$,
\item[{\it Elements in $\Omega_3$.}]
$g_3 := M_{\gA} \tilde{T}(a_1,a_2,b_1,b_2) U (c_1, c_2 )$,
\item[{\it Elements in $\Omega_4$.}]
$g_4 := M_{\gA} \tilde{T}(a_1,a_2,b_1,b_2) U (c_1, c_2 ) M_{\gW} U (d_1, d_2 )$.
\end{description}
Regarding $a_1$, $a_2$, $b_1$, $b_2$, $c_1$, $c_2$, $d_1$ and $d_2$ as variables, we set $\mathcal{G}_Q = \{ g_i \mid 1 \leq i \leq 4 \}$ in Main algorithm and Algorithm \ref{alg:AutoN1}. 
Moreover, we set $\mathcal{P}_Q = \{ a_1 a_2 - 1, (b_1^2 - \epsilon b_2^2) t - 1 , r v - 1\}$, where $t$, $r$ and $v$ are extra variables.
If the $3$-rd power map is surjective over $K$, we may assume $\mathcal{P}_Q = \{ a_1 a_2 - 1, (b_1^2 - \epsilon b_2^2) t - 1, r -1 \}$.

\paragraph{Elements in $\tilde\gO_{\varphi}(K)$ in (Dege) case}
Let $Q =2 y w + z^2$, and $\varphi$ its coefficient matrix.
Recall from Subsection \ref{subsec:orthoDege} that we have the Bruhat decomposition of the orthogonal similitude group associated to $\varphi$, say
\begin{eqnarray}
 \tilde\gO_{\varphi}(K)= (\gA \tilde \gT \gU \sqcup \gA \tilde \gT \gU M_{\rm W} \gU) \gV ,\label{eq:BruhatDege}
\end{eqnarray}
where $\gA$, $\tilde\gT$, $\gU$, $M_{\rm W}$ and $\gV$ are the same as in Subsection \ref{subsec:orthoDege}.
Here writing $\gA = \{ 1_4, M_{\gA}:= \mathrm{diag}(1,1,-1,1) \}$, it follows from \eqref{eq:BruhatDege} that we have $\tilde\gO_{\varphi}(K) = \bigcup_{i=1}^4 \Omega_i$, where
\begin{equation*}
\begin{array}{llll}
\Omega_1 := \tilde \gT \gU \gV, & \Omega_2 := \tilde \gT \gU M_{\rm W} \gU \gV, & \Omega_3 := M_{\gA} \tilde \gT \gU \gV, & \Omega_4 := M_{\gA} \tilde \gT \gU M_{\rm W} \gU \gV .
\end{array}
\end{equation*}
Put $\tilde{T}(a_1,a_2,a_3) = \diag(1,a_1 a_3,a_3,a_2 a_3)$,
\[
U (b) =
\begin{pmatrix}
1&0&0&0\\
0&1&b&-b^2/2\\
0&0&1&-b\\
0&0&0&1
\end{pmatrix}, \quad
U (c) =
\begin{pmatrix}
1&0&0&0\\
0&1&c&-c^2/2\\
0&0&1&-c\\
0&0&0&1
\end{pmatrix},
\]
\[
V (d, e_1, e_2, e_3) =
\begin{pmatrix}
d&e_1&e_2&e_3\\
0&1&0&0\\
0&0&1&0\\
0&0&0&1
\end{pmatrix},
\]
where $a_1$, $a_2$, $a_3$, $b$, $c$, $d$, $e_1$, $e_2$ and $e_3$ are elements in $K$ with $a_1 a_2=1$, $a_3 \neq 0$ and $d \neq 0$.
One can verify that elements of $\Omega_{i}$ for each $1 \leq i \leq 4$ are written as follows:
\begin{description}
\item[{\it Elements in $\Omega_1$.}]
$g_1 := \tilde{T} (a_1,a_2, a_3) U (b) V (d, e_1, e_2, e_3)$,
\item[{\it Elements in $\Omega_2$.}]
$g_2 := \tilde{T} (a_1,a_2, a_3) U (b) M_{\rm W} U (c) V (d, e_1, e_2, e_3)$,
\item[{\it Elements in $\Omega_3$.}]
$g_3 := M_{\rm A} \tilde{T} (a_1,a_2, a_3) U (b) V (d, e_1, e_2, e_3)$,
\item[{\it Elements in $\Omega_4$.}]
$g_4 := M_{\rm A} \tilde{T}(a_1,a_2, a_3) U (b) M_{\rm W} U (c) V (d, e_1, e_2, e_3)$.
\end{description}
Regarding $a_1$, $a_2$, $a_3$, $b$, $c$, $d$, $e_1$, $e_2$ and $e_3$ as variables, we set $\mathcal{G}_Q = \{ g_i \mid 1 \leq i \leq 4 \}$ in Main algorithm and Algorithm \ref{alg:AutoN1}. 
Moreover, we set $\mathcal{P}_Q = \{ a_1 a_2 - 1, a_3 s - 1, d t - 1, r v-1 \}$, where $s$, $t$, $r$ and $v$ are extra variables.
If the $3$-rd power map is surjective over $K$, we may assume $\mathcal{P}_Q = \{ a_1 a_2 - 1, a_3 s - 1, d t - 1, r -1 \}$.

\begin{rem}
The Bruhat decomposition of the orthogonal group for each quadratic form has some {\it redundant} variables, and thus we have to find and remove such variables;
otherwise the existence of such variables causes that the dimension of the ideal we construct in Main algorithm can be larger than 0. 
We describe this for the case of {\bf (N1)}.
As we stated above, the {\bf (N1)} case has the $8$ patterns $\Omega_i$ for $1 \leq i \leq 8$.
We here write down an element in $\Omega_1 = \tilde \gT \gU \gU$.
In this case, using the fact $\mathrm{U} \mathrm{U} = \mathrm{U}$, elements in $\Omega_1$ are written as $g_1 := \tilde{T} (a_1,a_2, b_1,b_2,c) U (d_1, d_2 )$ for some $a_1$, $a_2$, $b_1$, $b_2$, $c$, $d_1$ and $d_2$ in $K$ with $a_1 a_2 =1$, $b_1 b_2 =1$ and $c \neq 0$.
However, if one does not use the fact $\mathrm{U} \mathrm{U} = \mathrm{U}$, elements in $\Omega_1$ can be represented as $\tilde{T} (a_1,a_2, b_1,b_2,c) U (d_1, d_2 ) U (e_1, e_2 )$ for some $e_1$ and $e_2$.
Here we have
\[
U(d_1,d_2) U(e_1,e_2) =
\begin{pmatrix}
1&d_1+e_1&d_2+e_2&-d_1 d_2-d_1 e_2-d_2 e_1-e_1 e_2\\
0&1&0&-d_2-e_2\\
0&0&1&-d_1-e_1\\
0&0&0&1
\end{pmatrix} .
\]
Since $d_1 + e_1$ and $d_2 + e_2$ can take independently all elements in $K$, we may assume $e_1=0$ and $e_2=0$.
Regarding unknown entries as variable, we regard only $d_1$ and $d_2$ as variables and assume $e_1=0$ and $e_2=0$ (if one regards $e_1$ as a variables, $d_1$ and $e_1$ are redundant).
In case of each $\Omega_i$ with $i \neq1$, we have the same argument.
\end{rem}

\subsection{Determining the structure of automorphism groups}\label{subsec:isom}

This subsection describes our method to determine the finite group structure of a given subgroup of the the general linear group of degree $4$.
Main algorithm in Subsection \ref{subsec:compaut} computes $\mathrm{Aut}_K (C)$ of a nonhyperelliptic superspecial curve $C$ of genus $4$ over $K$ as a subset of $\mathrm{GL}_4 (K)$.
Each computed subset is finite even if $K$ is an algebraically closed field.
(Indeed, Stichtenoth \cite{Stichtenoth} showed that the automorphism group of a curve of genus $g \ge 2$ defined over a field of characteristic $p>0$ is finite and has cardinality bounded by $16 g^4$.)
Hence our next aim is to determine the finite group structure of each of computed automorphism groups.
For this, we take the following:
For each of computed automorphism groups,
\begin{enumerate}
\item Using a computer algebra system, find the finite group isomorphic to the automorphism group.
Specifically, we choose a candidate for the finite group isomorphic to the automorphism group, and then test whether the automorphism group and the candidate are isomorphic or not.
In our case, we use Magma's built-in function \texttt{IsIsomorphic()}, which judges whether two groups are isomorphic or not. 
\item Based on the result by Magma, determine the finite field isomorphic to the automorphism group.
Specifically, we construct an isomorphism from the automorphism group to the group obtained in (1).
\end{enumerate}
For our choice of the candidate in (1), see Remark \ref{rem:isogroup} below.

\begin{rem}\label{rem:isogroup}
Based on our knowledge to the theory of finite groups, we determine the finite group isomorphic to the automorphism group.
Since the orders of some automorphism groups can be large, we try to find such a finite group by our heuristic.
(cf.\ Most of finite groups with small orders are classified, and hence we can find the isomorphic group from known class of finite groups with small orders.)
The candidates we choose are the cyclic group ${\rm C}_{t}$, the dihedral group ${\rm D}_t$, the symmetric group ${\rm S}_t$, the alternating group ${\rm A}_t$, and their direct products.
Given an automorphism group $G$ with order $n$, proceed with the following steps:
\begin{description}
	\item[{\it Step 1.}] Test whether the group $G$ is isomorphic to the $n$-th cyclic group ${\rm C}_n$. (Note that if $n$ is a prime number $p^{\prime}$, then $G \cong {\rm C}_{p^{\prime}}$.)
	\item[{\it Step 2.}] If $n$ is even, say $n=2k_1$ for some $k_1 \in \Z_{\ge 0}$, then test whether $G \cong {\rm D}_{k_1}$ or not.
	\item[{\it Step 3.}] If $n$ is a factorial number, say $n={k_2} !$ for some $k_2 \in \Z_{\ge 0}$, then decide whether $G \cong {\rm S}_{k_2}$ or not.
	\item[{\it Step 4.}] If $n = {k_3}!/2$ for some $k_3 \in \Z_{\ge 0}$, then test whether $G \cong {\rm A}_{k_3}$ or not.
	\item[{\it Step 5.}] If $G$ is not isomorphic to any finite group in Steps 1 -- 4, then test whether $G$ is isomorphic to a product of some groups (e.g., ${\rm C}_{n_1} \times {\rm C}_{n_2}$ with $n=n_1 n_2$).
\end{description}
Note that this method is not guaranteed to find the finite group isomorphic to any automorphism group, but we succeeded in finding the finite groups isomorphic to automorphism groups of the superspecial curves in Theorems \ref{theo:N1Aut}, \ref{theo:N2Aut}, \ref{theo:DegenerateAut} and \ref{theo:AlcAut} in Section \ref{section4}.
\end{rem}

\section{Main Results}\label{section4}
This section shows our main results on computing the automorphism groups of nonhyperelliptic superspecial curves of genus $4$ over $\mathbb{F}_{11}$.
Specifically, by using algorithms given in Section \ref{section3}, we shall compute the automorphism groups as subgroups of the general linear group of degree $4$ over $\mathbb{F}_{11}$ (or $\overline{\mathbb{F}_{11}}$).
Moreover, we determine the group structure of each automorphism group.

Before proving main theorems, let us recall defining equations of nonhyperelliptic superspecial curves of genus $4$ over $\mathbb{F}_{11}$, which are given in \cite[Section 4]{KH17}.
There are $30$ nonhyperelliptic superspcial curves of genus $4$ over $\F_{11}$ up to isomorphism over $\F_{11}$, and the number of their $\overline{\mathbb{F}_{11}}$-isomorphism classes is $9$. 
Each isomorphism class is defined by an irreducible quadratic form $Q$ and an irreducible cubic form $P$ in $\mathbb{F}_{11}[x,y,z,w]$, where $Q$ has the following three types:
{\bf (N1)} $Q^{\rm (N1)}=2 x w + 2 y z$, {\bf (N2)} $Q^{\rm (N2)} = 2 x w + y^2 - \epsilon z^2$ with $\epsilon \in \F_{11}^{\times} \smallsetminus (\F_{11}^{\times})^2$, and {\bf (Dege)} $Q^{\rm (Dege)} = 2 y w + z^2$.
In the following, we state all pairs of $Q$ and $P$, which define the $\mathbb{F}_{11}$-isomorphism classes of nonhyperelliptic superspcial curves $V(Q,P)$ of genus $4$ over $\F_{11}$.

\paragraph{Recall: Defining equations of nonhyperelliptic superspecial curves of genus $4$ over $\mathbb{F}_{11}$}

\begin{description}
\item[{\bf Case of (N1):}] The $8$ isomorphism classes $C_i^{\rm (N1)}=V (Q^{\rm (N1)}, P_i^{\rm (N1)})$ for $1 \leq i \leq 8$, where $Q^{\rm (N1)} = 2 x w + 2 y z$ and
\begin{eqnarray}
P_1^{({\rm N1})}&= & x^2 y + x^2 z + 2 y^2 z + 5 y^2 w + 9 y z^2 + y z w + 4 z^3 + 3 z^2 w + 10 z w^2 + w^3, \nonumber \\[-0.5mm]
P_2^{({\rm N1})}&= & x^2 y + x^2 z + y^3 + y^2 z + 7 y z^2 + 4 y w^2 + 2 z^3 + 9 z w^2, \nonumber \\[-0.5mm]
P_3^{({\rm N1})}&= & x^2 y + x^2 z + y^3 + 8 y^2 z + 3 y z^2 + 10 y w^2 + 10 z^3 + 10 z w^2, \nonumber \\[-0.5mm]
P_4^{({\rm N1})}&= & x^2 y + x^2 z + y^3 + 9 y^2 z + 2 y^2 w + 3 y z^2 + 3 y z w + 4 y w^2 + 10 z^3 + 2 z^2 w \nonumber \\[-0.5mm]
& &+ 6 z w^2, \nonumber \\[-0.5mm]
P_5^{({\rm N1})}&= & x^2 y + x^2 z + x z^2 + 10 y^2 w + 9 y z^2 + 9 y w^2 + 8 z^3 + 8 z^2 w + 8 z w^2 + 3 w^3, \nonumber \\[-0.5mm]
P_6^{({\rm N1})}&= & x^2 y + x^2 z + x z^2 + 9 y^2 z + 5 y^2 w + y z w + 8 y w^2 + 3 z^3 + 9 z^2 w + 2 z w^2 + 5 w^3, \nonumber \\[-0.5mm]
P_7^{({\rm N1})}&= & x^2 y + x^2 z + x z^2 + 4 y^3 + 2 y^2 z + 10 y^2 w + 3 y z^2 + 8 y z w + 8 y w^2 + 8 z^3 \nonumber \\[-0.5mm] 
& &  + 7 z^2 w + 7 z w^2 + 4 w^3, \nonumber \\[-0.5mm]
P_8^{({\rm N1})}&= & x^2 y + x^2 z + x z^2 + 9 y^3 + 6 y^2 z + 5 y^2 w + 8 y z^2 + 5 y z w + 2 y w^2 + z^3 + 2 z^2 w  \nonumber \\[-0.5mm] 
& & + 7 z w^2 + w^3. \nonumber
\end{eqnarray}

\item[{\bf Case of (N2):}] The $5$ isomorphism classes $C_i^{\rm (N2)}:=V (Q^{\rm (N2)}, P_i^{\rm (N2)})$ for $1 \leq i \leq 5$, where $Q^{\rm (N2)} = 2 x w + y^2 - \epsilon z^2$ with $\epsilon \in \F_{11}^{\times} \smallsetminus (\F_{11}^{\times})^2$ and
\begin{eqnarray}
P_1^{({\rm N2})}& = & x^2 y + x^2 z + x y^2 + 9 x z^2 + 6 y^3 + y^2 z + 5 y^2 w + 3 y z^2 + 9 y w^2 + 8 z^3 + z^2 w \nonumber \\[-0.5mm] 
& &  + 9 z w^2 + 6 w^3, \nonumber \\[-0.5mm]
P_2^{({\rm N2})}&= & x^2 z + 5 y^3 + 4 z w^2, \nonumber \\[-0.5mm]
P_3^{({\rm N2})}&= & x^2 y + x^2 z + 9 y^3 + 8 y^2 z + 2 y z^2 + 4 y w^2 + 9 z^3 + 4 z w^2, \nonumber \\[-0.5mm]
P_4^{({\rm N2})}&= & 8 x^2 y + 2 x^2 z + y^3 + 8 y^2 z + 6 y^2 w + 9 y z^2 + 2 y z w + 5 y w^2 + 9 z^3 + z^2 w \nonumber \\[-0.5mm]
& & + 4 z w^2 + w^3, \nonumber \\[-0.5mm]
P_5^{({\rm N2})}&= & 6 x^2 y + 4 x^2 z + 6 x y^2 + 10 x z^2 + 10 y^3 + 4 y^2 z + 3 y^2 w + 8 y z^2 + 6 y z w + 9 y w^2  \nonumber \\[-0.5mm]
& & + 10 z^3 + z^2 w + z w^2 + 9 w^3. \nonumber
\end{eqnarray}

\item[{\bf Case of (Dege):}] The $17$ isomorphism classes $C_i^{\rm (Dege)}:=V (Q^{\rm (Dege)}, P_i^{\rm (Dege)})$ for $1 \leq i \leq 17$, where $Q^{\rm (Dege)} = 2 y w + z^2$ and
\begin{eqnarray}
P_1^{({\rm Dege})}&= & x^3+y^3+w^3, \nonumber \\[-0.5mm]
P_2^{({\rm Dege})}&= & x^3+y^3+2 w^3, \nonumber \\[-0.5mm]
P_3^{({\rm Dege})}&= & x^3+y^3+z^3+5 w^3, \nonumber \\[-0.5mm]
P_4^{({\rm Dege})}&= & x^3+ x w^2+y^3 , \nonumber \\[-0.5mm]
P_5^{({\rm Dege})}&= & x^3+2 x w^2 +y^3, \nonumber \\[-0.5mm]
P_6^{({\rm Dege})}&= & x^3+x z w+y^3+7 z^3+ w^3, \nonumber \\[-0.5mm]
P_7^{({\rm Dege})}&= & x^3+x w^2+x y z+y^3+5 z^3 + 4 w^3, \nonumber \\[-0.5mm]
P_8^{({\rm Dege})}&= & x^3+6 x w^2 +x y z+y^3+8 z^3+8 w^3 , \nonumber \\[-0.5mm]
P_9^{({\rm Dege})}&= & x^3+5 y^3+2 y z^2+z^3 +z w^2 +4 w^3, \nonumber \\[-0.5mm]
P_{10}^{({\rm Dege})}&= & x^3+y^3+8 y z^2 + z^2 w+2 w^3, \nonumber \\[-0.5mm]
P_{11}^{({\rm Dege})} &= & x^3+2 y^3+2 y z^2+4 z^3 +  z^2 w+3 w^3, \nonumber \\[-0.5mm]
P_{12}^{({\rm Dege})} &= & x^3+2 y^3+4 y z^2 + z^2 w +10 w^3, \nonumber \\[-0.5mm]
P_{13}^{({\rm Dege})}&= & x^3+2 y^3+4 y z^2+z^3+ z^2 w +z w^2 +7 w^3, \nonumber \\[-0.5mm]
P_{14}^{({\rm Dege})}&= & x^3+x y^2+7 x y z+8 x z^2+8 x z w +2 x w^2 +2 y^3+4 y z^2+z^3 + z^2 w+z w^2 \nonumber \\[-0.5mm] 
& & + 7 w^3, \nonumber \\[-0.5mm]
P_{15}^{({\rm Dege})} &= & x^3+5 y^3+3 y z^2+5 z^3 +z^2 w+z w^2 + 10 w^3, \nonumber \\[-0.5mm]
P_{16}^{({\rm Dege})}&= & x^3+6 y^3+2 y z^2+6 z^3 +z^2 w + z w^2+ 6 w^3, \nonumber \\[-0.5mm]
P_{17}^{({\rm Dege})}& =& x^3+10 y^3+6 y z^2+7 z^3 + z^2 w+z w^2. \nonumber 
\end{eqnarray}
\end{description}

\subsection{The automorphism groups of superspecial curves}\label{subsec:Aut}

In this subsection, we determine the automorphism groups of nonhyperelliptic superspecial curves of genus $4$ over the prime field $\mathbb{F}_{11}$.
As in Subsection \ref{subsec:isom}, we denote by ${\rm C}_n$, ${\rm D}_n$, ${\rm S}_n$ and ${\rm A}_n$ the cyclic group of degree $n$, the dihedral group of degree $n$, the symmetric group of degree $n$, and the alternating group of degree $n$ respectively, for each $n$.
For two groups $G$ and $H$, we also denote by $G \times H$ their direct product.

\subsubsection{Case of (N1) over the prime field $\mathbb{F}_{11}$}\label{subsubsec:N1}

First, we determine the group structure of $\mathrm{Aut}_{\mathbb{F}_{11}}(C_i^{\rm (N1)})$, where $C_i^{\rm (N1)} = V ( Q^{\rm (N1)}, P_i^{\rm (N1)})$ with $Q^{\rm (N1)} = 2 x w + 2 y z$ for $1 \leq i \leq 8$.

\begin{thm}\label{theo:N1Aut}
Let $C_i^{(N1)} = V ( Q^{\rm (N1)}, P_i^{\rm (N1)})$ denote the superspecial curve of genus $4$ over $\mathbb{F}_{11}$ defined by $Q^{\rm (N1)}$ and $P_i^{\rm (N1)}$ for each $1 \leq i \leq 8$.
Then we have the following isomorphisms:
\begin{tabular}{clclcl}
$(1)$ & $\Aut_{\mathbb{F}_{11}} ( C_1^{\rm (N1)} ) \cong {\rm C}_6$,  & 
$(2)$ & $\Aut_{\mathbb{F}_{11}} ( C_2^{\rm (N1)} ) \cong {\rm C}_2$, &
$(3)$ & $\Aut_{\mathbb{F}_{11}} ( C_3^{\rm (N1)} ) \cong {\rm D}_4$, \\[2mm]
$(4)$ & $\Aut_{\mathbb{F}_{11}} ( C_4^{\rm (N1)} ) \cong {\rm C}_2$, &
$(5)$ & $\Aut_{\mathbb{F}_{11}} ( C_5^{\rm (N1)} ) \cong {\rm C}_3$, & 
$(6)$ & $\Aut_{\mathbb{F}_{11}} ( C_6^{\rm (N1)} ) \cong {\rm C}_2 \times {\rm C}_2$, \\[2mm]
$(7)$ & $\Aut_{\mathbb{F}_{11}} ( C_7^{\rm (N1)} ) \cong {\rm D}_6$, & 
$(8)$ & $\Aut_{\mathbb{F}_{11}} ( C_8^{\rm (N1)} ) \cong {\rm S}_4$. & &
\end{tabular}
\end{thm}

\begin{proof}
To simplify the notation, we set $Q:=Q^{\rm (N1)}$ and $P_i:=P_i^{\rm (N1)}$ for $1 \leq i \leq 8$ through this proof.
Putting $G_i := \{\, g \in \tilde{\gO}_{\varphi}(\mathbb{F}_{11}) \mid g \cdot P_i \equiv r P_i \bmod{Q} \mbox{ for some } r \in \mathbb{F}_{11}^{\times} \,\}$, we have $\Aut_{\mathbb{F}_{11}}(C_i^{\rm (N1)} )=G_i /{\sim}$, where $\varphi$ denotes the matrix associated to the quadratic form $Q$, and we write $g \sim cg$ for $g \in \tilde{\gO}_{\varphi}(\mathbb{F}_{11})$ and $c \in \mathbb{F}_{11}^{\times}$.
We also recall that $g \cdot P := P ( (x, y, z, w) \cdot {}^t g )$ for $g \in \mathrm{GL}_{4} (\mathbb{F}_{11})$ and a cubic form $P \in \mathbb{F}_{11}[x,y,z,w]$.
By Proposition \ref{prop:N1Aut} in Subsection \ref{subsubsec:compN1}, we have explicit generators of $G_i / {\sim}$ for each $i$.

Here, we show only the statement (1) since the other cases (2) -- (8) are proved in ways similar to (1).
Proposition \ref{prop:N1Aut} says that $G_1 / {\sim}$ is generated by
\[
\begin{pmatrix}
2&4&5&1\\
3&6&5&1\\
7&7&1&10\\
6&6&10&1
\end{pmatrix},
\]
whose order is $6$, and hence $G_1 / {\sim}$ is isomorphic to ${\rm C}_6$.
\end{proof}

In Table \ref{tab:N1Aut}, we summarize the results in Theorem \ref{theo:N1Aut}.

\renewcommand{\arraystretch}{1.3}
\begin{table}[H]
\begin{center}
\begin{tabular}{ccc} \hline
Superspecial curves $C$ over $\F_{11}$ & $\Aut_{\F_{11}}{(C)}$ & $\# \Aut_{\F_{11}}{(C)}$ \\ \hline
$C_1^{\rm (N1)} = V(Q^{({\rm N1})},P_1^{({\rm N1})})$  & ${\rm C}_6$ & $6$\\
$C_2^{\rm (N1)} = V(Q^{({\rm N1})},P_2^{({\rm N1})})$  & ${\rm C}_2$ & $2$\\
$C_3^{\rm (N1)} = V(Q^{({\rm N1})},P_3^{({\rm N1})})$ & ${\rm D}_4$ & $8$ \\
$C_4^{\rm (N1)} = V(Q^{({\rm N1})},P_4^{({\rm N1})})$ & ${\rm C}_2$ & $2$ \\
$C_5^{\rm (N1)} = V(Q^{({\rm N1})},P_5^{({\rm N1})})$ & ${\rm C}_3$ & $3$ \\
$C_6^{\rm (N1)} = V(Q^{({\rm N1})},P_6^{({\rm N1})})$  & ${\rm C}_2 \times {\rm C}_2$ & $4$\\
$C_7^{\rm (N1)} = V(Q^{({\rm N1})},P_7^{({\rm N1})})$ & ${\rm D}_6$ & $12$ \\
$C_8^{\rm (N1)} = V(Q^{({\rm N1})},P_8^{({\rm N1})})$ & ${\rm S}_4$ & $24$ \\ \hline
\end{tabular}
\caption{The automorphism groups $\mathrm{Aut}_{\mathbb{F}_{11}} (C_i^{\rm (N1)})$ of the superspecial curves $C_i^{{\rm (N1)}}= V (Q^{\rm (N1)}, P_i^{\rm (N1)})$ over $\mathbb{F}_{11}$ for $1 \leq i \leq 8$ in the case of {\bf (N1)}.
Here $Q^{\rm (N1)} = 2 x w + 2 y z$, and each $P_i^{\rm (N1)}$ is defined at the beginning of this section (Section \ref{section4}).}\label{tab:N1Aut}
\end{center}
\end{table}

\subsubsection{Case of (N2) over the prime field $\mathbb{F}_{11}$}\label{subsubsec:N2}

Next, we determine the group structure of $\mathrm{Aut}_{\mathbb{F}_{11}}(C_i^{\rm (N2)})$, where $C_i^{\rm (N2)} = V ( Q^{\rm (N2)}, P_i^{\rm (N2)})$ with $Q^{\rm (N2)} = 2 x w + y^2 - \epsilon z^2$ and $\epsilon \in \F_{11}^{\times} \smallsetminus (\F_{11}^{\times})^2$ for $1 \leq i \leq 5$.

\begin{thm}\label{theo:N2Aut}
Let $C_i^{\rm (N2)} = V ( Q^{\rm (N2)}, P_i^{\rm (N2)})$ denote the superspecial curve of genus $4$ over $\mathbb{F}_{11}$ defined by $Q^{\rm (N2)}$ and $P_i^{\rm (N2)}$ for each $1 \leq i \leq 5$.
Then we have the following isomorphisms:
\begin{tabular}{clclcl}
$(1)$ & $\Aut_{\mathbb{F}_{11}} ( C_1^{\rm (N2)} ) \cong {\rm D}_6$,  & 
$(2)$ & $\Aut_{\mathbb{F}_{11}} ( C_2^{\rm (N2)} ) \cong {\rm C}_2 \times {\rm C}_2$, &
$(3)$ & $\Aut_{\mathbb{F}_{11}} ( C_3^{\rm (N2)} ) \cong {\rm C}_2 \times {\rm C}_2$, \\[2mm]
$(4)$ & $\Aut_{\mathbb{F}_{11}} ( C_4^{\rm (N2)} ) \cong {\rm C}_4$, &
$(5)$ & $\Aut_{\mathbb{F}_{11}} ( C_5^{\rm (N2)} ) \cong {\rm C}_6$. &  &
\end{tabular}
\end{thm}

\begin{proof}
To simplify the notation, we set $Q:=Q^{\rm (N2)}$ and $P_i:=P_i^{\rm (N2)}$ for $1 \leq i \leq 5$ through this proof.
Putting $G_i := \{\, g \in \tilde{\gO}_{\varphi}(\mathbb{F}_{11}) \mid g \cdot P_i \equiv r P_i \bmod{Q} \mbox{ for some } r \in \mathbb{F}_{11}^{\times} \,\}$, we have $\Aut_{\mathbb{F}_{11}}(C_i^{\rm (N2)} )=G_i /{\sim}$, where $\varphi$ denotes the matrix associated to $Q$, and we write $g \sim cg$ for $g \in \tilde{\gO}_{\varphi}(\mathbb{F}_{11})$ and $c \in \mathbb{F}_{11}^{\times}$.
We also recall that $g \cdot P := P ( (x, y, z, w) \cdot {}^t g )$ for $g \in \mathrm{GL}_{4} (\mathbb{F}_{11})$ and a cubic form $P \in \mathbb{F}_{11}[x,y,z,w]$.
By Proposition \ref{prop:N2Aut} in Subsection \ref{subsubsec:compN2}, we have explicit generators of $G_i / {\sim}$ for each $i$.

Here, we show only the statement (1) for the same reason as that in the proof of Theorem \ref{theo:N1Aut}.
Proposition \ref{prop:N2Aut} says that $G_1 / {\sim}$ is generated by
\[
a:=
\begin{pmatrix}
1&6&9&5\\
10&8&6&6\\
1&8&1&1\\
6&10&9&1
\end{pmatrix},
\quad \text{and} \quad
b:=
\begin{pmatrix}
1&0&8&5\\
10&1&3&6\\
1&0&7&1\\
6&1&2&1
\end{pmatrix},
\]
whose orders are $2$ and $6$ respectively.
The homomorphism given by $a \mapsto (1, 6) (2, 5) (3, 4)$ and $b \mapsto (1, 2, 3, 4, 5, 6)$ defines an isomorphism between $G_1 / \sim$ and ${\rm D}_6$, where we identify ${\rm D}_6$ with the subgroup of $\mathrm{S}_6$ generated by the permutations $(1,6)(2,5)(3,4)$ and $(1,2,3,4,5,6)$.
\end{proof}

Table \ref{tab:N2Aut} summarizes the results in Theorem \ref{theo:N2Aut}.

\renewcommand{\arraystretch}{1.3}
\begin{table}[H]
\begin{center}
\begin{tabular}{ccc} \hline
Superspecial curves $C$ over $\F_{11}$ & $\Aut_{\F_{11}}{(C)}$ & $\# \Aut_{\F_{11}}{(C)}$ \\ \hline
$C_1^{\rm (N2)} = V(Q^{({\rm N2})},P_1^{({\rm N2})})$ & ${\rm D}_6$ & $12$ \\
$C_2^{\rm (N2)} = V(Q^{({\rm N2})},P_2^{({\rm N2})})$ & ${\rm C}_2 \times {\rm C}_2$ & $4$ \\
$C_3^{\rm (N2)} = V(Q^{({\rm N2})},P_3^{({\rm N2})})$ & ${\rm C}_2 \times {\rm C}_2$ & $4$ \\
$C_4^{\rm (N2)} = V(Q^{({\rm N2})},P_4^{({\rm N2})})$ & ${\rm C}_4$ & $4$ \\
$C_5^{\rm (N2)} = V(Q^{({\rm N2})},P_5^{({\rm N2})})$ & ${\rm C}_6$ & $6$\\ \hline
\end{tabular}
\caption{The automorphism groups $\mathrm{Aut}_{\mathbb{F}_{11}} (C_i^{\rm (N2)})$ of the superspecial curves $C_i^{{\rm (N2)}}= V (Q^{\rm (N2)}, P_i^{\rm (N2)})$ over $\mathbb{F}_{11}$ for $1 \leq i \leq 5$ in the case of {\bf (N2)}.
Here $Q^{\rm (N2)} = 2 x w + y^2 - \epsilon z^2$ with $\epsilon \in \F_{11}^{\times} \smallsetminus (\F_{11}^{\times})^2$, and each $P_i^{\rm (N2)}$ is defined at the beginning of this section (Section \ref{section4}).}\label{tab:N2Aut}
\end{center}
\end{table}

\subsubsection{Case of (Dege) over the prime filed $\mathbb{F}_{11}$}\label{subsubsec:Dege}

We determine the group structure of $\mathrm{Aut}_{\mathbb{F}_{11}}(C_i^{\rm (Dege)})$, where $C_i^{\rm (Dege)} = V ( Q^{\rm (Dege)}, P_i^{\rm (Dege)})$ with $Q^{\rm (Dege)} = 2 y w + z^2$ for $1 \leq i \leq 17$.

\begin{thm}\label{theo:DegenerateAut}
Let $C_i^{{\rm (Dege)}} = V ( Q^{\rm (Dege)}, P_i^{\rm (Dege)})$ denote the superspecial curve of genus $4$ over $\mathbb{F}_{11}$ defined by $Q^{\rm (Dege)}$ and $P_i^{\rm (Dege)}$ for each $1 \leq i \leq 17$.
Then we have the following isomorphisms:

\begin{tabular}{clcl}
$(1)$ & $\Aut_{\mathbb{F}_{11}} ( C_1^{\rm (Dege)} ) \cong {\rm C}_2 \times {\rm C}_2$,  & 
$(2)$ & $\Aut_{\mathbb{F}_{11}} ( C_2^{\rm (Dege)} ) \cong {\rm C}_2 \times {\rm C}_2$, \\[2mm]
$(3)$ & $\Aut_{\mathbb{F}_{11}} ( C_3^{\rm (Dege)} ) \cong {\rm C}_2 \times {\rm C}_2$, &
$(4)$ & $\Aut_{\mathbb{F}_{11}} ( C_4^{\rm (Dege)} ) \cong {\rm C}_2$, \\[2mm]
$(5)$ & $\Aut_{\mathbb{F}_{11}} ( C_5^{\rm (Dege)} ) \cong {\rm C}_2$, & 
$(6)$ & $\Aut_{\mathbb{F}_{11}} ( C_6^{\rm (Dege)} ) \cong \{ e \}$, \\[2mm]
$(7)$ & $\Aut_{\mathbb{F}_{11}} ( C_7^{\rm (Dege)} ) \cong {\rm C}_2$, &
$(8)$ & $\Aut_{\mathbb{F}_{11}} ( C_8^{\rm (Dege)} ) \cong \{ e \}$, \\[2mm]
$(9)$ & $\Aut_{\mathbb{F}_{11}} ( C_9^{\rm (Dege)} ) \cong {\rm C}_6$, &
$(10)$ & $\Aut_{\mathbb{F}_{11}} ( C_{10}^{\rm (Dege)} ) \cong {\rm D}_6$, \\[2mm]
$(11)$ & $\Aut_{\mathbb{F}_{11}} ( C_{11}^{\rm (Dege)} ) \cong {\rm D}_6$,  & 
$(12)$ & $\Aut_{\mathbb{F}_{11}} ( C_{12}^{\rm (Dege)} ) \cong {\rm S}_4$, \\[2mm]
$(13)$ & $\Aut_{\mathbb{F}_{11}} ( C_{13}^{\rm (Dege)} ) \cong {\rm C}_4$, &
$(14)$ & $\Aut_{\mathbb{F}_{11}} ( C_{14}^{\rm (Dege)} ) \cong {\rm C}_2$, \\[2mm]
$(15)$ & $\Aut_{\mathbb{F}_{11}} ( C_{15}^{\rm (Dege)} ) \cong {\rm C}_6$, &  
$(16)$ & $\Aut_{\mathbb{F}_{11}} ( C_{16}^{\rm (Dege)} ) \cong {\rm C}_3$, \\[2mm]
$(17)$ & $\Aut_{\mathbb{F}_{11}} ( C_{17}^{\rm (Dege)} ) \cong {\rm D}_4$. & &
\end{tabular}
\end{thm}

\begin{proof}
To simplify the notation, we set $Q:=Q^{\rm (Dege)}$ and $P_i:=P_i^{\rm (Dege)}$ for $1 \leq i \leq 17$ through this proof.
Putting $G_i := \{\, g \in \tilde{\gO}_{\varphi}(\mathbb{F}_{11}) \mid g \cdot P_i \equiv r P_i \bmod{Q} \mbox{ for some } r \in \mathbb{F}_{11}^{\times} \,\}$, we have $\Aut_{\mathbb{F}_{11}}(C_i^{\rm (Dege)} )=G_i /{\sim}$, where $\varphi$ denotes the matrix associated to $Q$, and we write $g \sim cg$ for $g \in \tilde{\gO}_{\varphi}(\mathbb{F}_{11})$ and $c \in \mathbb{F}_{11}^{\times}$.
We also recall that $g \cdot P := P ( (x, y, z, w) \cdot {}^t g )$ for $g \in \mathrm{GL}_{4} (\mathbb{F}_{11})$ and a cubic form $P \in \mathbb{F}_{11}[x,y,z,w]$.
By Proposition \ref{prop:DegeAut} in Subsection \ref{subsubsec:compDege}, we have explicit generators of $G_i / {\sim}$ for each $i$.

Here, we show only the statement (1) for the same reason as that in the proof of Theorem \ref{theo:N1Aut}.
Proposition \ref{prop:DegeAut} says that $G_1 / {\sim}$ is generated by
\[
a:={\diag}(1,1,-1,1),
\quad \text{and} \quad
b:=
\begin{pmatrix}
1&0&0&0\\
0&0&0&1\\
0&0&1&0\\
0&1&0&0
\end{pmatrix},
\]
whose orders are $2$ and $2$ respectively.
The homomorphism given by $a \mapsto (1, 2)$ and $b \mapsto (3,4)$ defines an isomorphism between $G_1 / \sim$ and ${\rm C}_2 \times {\rm C}_2$, where we identify ${\rm C}_2 \times {\rm C}_2$ with the subgroup of $\mathrm{S}_4$ generated by the permutations $(1,2)$ and $(3,4)$.
\end{proof}

In Table \ref{tab:DegeAut}, we summarize the results in Theorem \ref{theo:DegenerateAut}.

\renewcommand{\arraystretch}{1.3}
\begin{table}[H]
\begin{center}
\begin{tabular}{ccc} \hline
Superspecial curves $C$ over $\F_{11}$ & $\Aut_{\F_{11}}{(C)}$ & $\# \Aut_{\F_{11}}{(C)}$ \\ \hline
$C_1^{\rm (Dege)} = V(Q^{({\rm Dege})},P_1^{({\rm Dege})})$ & ${\rm C}_2 \times {\rm C}_2$ & $4$ \\
$C_2^{\rm (Dege)} = V(Q^{({\rm Dege})},P_2^{({\rm Dege})})$ & ${\rm C}_2 \times {\rm C}_2$ & $4$ \\
$C_3^{\rm (Dege)} = V(Q^{({\rm Dege})},P_3^{({\rm Dege})})$ & ${\rm C}_2 \times {\rm C}_2$ & $4$ \\
$C_4^{\rm (Dege)} = V(Q^{({\rm Dege})},P_4^{({\rm Dege})})$ & ${\rm C}_2$ & $2$ \\
$C_5^{\rm (Dege)} = V(Q^{({\rm Dege})},P_5^{({\rm Dege})})$ & ${\rm C}_2$ & $2$ \\
$C_6^{\rm (Dege)} = V(Q^{({\rm Dege})},P_6^{({\rm Dege})})$ & $\{e\}$ & $1$\\
$C_7^{\rm (Dege)} = V(Q^{({\rm Dege})},P_7^{({\rm Dege})})$ & ${\rm C}_2$ & $2$\\
$C_8^{\rm (Dege)} = V(Q^{({\rm Dege})},P_8^{({\rm Dege})})$ & $\{e\}$ & $1$ \\
$C_9^{\rm (Dege)} = V(Q^{({\rm Dege})},P_9^{({\rm Dege})})$ & ${\rm C}_6$ & $6$ \\
$C_{10}^{\rm (Dege)} = V(Q^{({\rm Dege})},P_{10}^{({\rm Dege})})$ & ${\rm D}_6$ & $12$\\
$C_{11}^{\rm (Dege)} = V(Q^{({\rm Dege})},P_{11}^{({\rm Dege})})$ & ${\rm D}_6$ & $12$ \\
$C_{12}^{\rm (Dege)} = V(Q^{({\rm Dege})},P_{12}^{({\rm Dege})})$ & ${\rm S}_4$ & $24$ \\
$C_{13}^{\rm (Dege)} = V(Q^{({\rm Dege})},P_{13}^{({\rm Dege})})$ & ${\rm C}_4$ & $4$ \\
$C_{14}^{\rm (Dege)} = V(Q^{({\rm Dege})},P_{14}^{({\rm Dege})})$ & ${\rm C}_2$ & $2$ \\
$C_{15}^{\rm (Dege)} = V(Q^{({\rm Dege})},P_{15}^{({\rm Dege})})$ & ${\rm C}_6$ & $6$ \\
$C_{16}^{\rm (Dege)} = V(Q^{({\rm Dege})},P_{16}^{({\rm Dege})})$ & ${\rm C}_3$ & $3$ \\
$C_{17}^{\rm (Dege)} = V(Q^{({\rm Dege})},P_{17}^{({\rm Dege})})$ & ${\rm D}_4$ & $8$ \\ \hline
\end{tabular}
\caption{The automorphism groups $\mathrm{Aut}_{\mathbb{F}_{11}} (C_i^{\rm (Dege)})$ of the superspecial curves $C_i^{{\rm (Dege)}}= V (Q^{\rm (Dege)}, P_i^{\rm (Dege)})$ over $\mathbb{F}_{11}$ for $1 \leq i \leq 17$ in the case of {\bf (Dege)}.
Here $Q^{\rm (Dege)} = 2 y w + z^2$, and each $P_i^{\rm (Dege)}$ is defined at the beginning of this section (Section \ref{section4}).}\label{tab:DegeAut}
\end{center}
\end{table}

\subsubsection{Necessary condition on the sum of the orders of automorphism groups}
Our main theorems (Theorems \ref{theo:N1Aut} -- \ref{theo:DegenerateAut}) are proved with the help of computational results given in Subsection \ref{subsec:comp}.
In this subsection, we shall confirm that our results satisfy a theoretically necessary condition.
Let $C_1, \ldots , C_k$ be all $\mathbb{F}_{11}$-isomorphism classes of superspecial curves of genus $4$ over $\mathbb{F}_{11}$ such that $C_i \cong C_j$ over the algebraic closure $\overline{\mathbb{F}_{11}}$ for all $1 \leq i < j \leq k$ .
It is known that the sum of the reciprocals of $\# \Aut_{\mathbb{F}_{11}}(C_i)$ for $1 \leq i \leq k$ is equal to $1$.
In Corollary \ref{cor:sumaut} below, we prove that all orders obtained with the help of our computational results satisfy this necessary condition.
This supports the truth of the main theorems (Theorems \ref{theo:N1Aut} -- \ref{theo:DegenerateAut}) obtained by computational results.

\begin{cor}\label{cor:sumaut}
Let $C_1, \ldots , C_k$ be all $\mathbb{F}_{11}$-isomorphism classes of nonhyperelliptic superspecial curves of genus $4$ over $\mathbb{F}_{11}$ such that $C_i \cong C_j$ over the algebraic closure $\overline{\mathbb{F}_{11}}$ for all $1 \leq i < j \leq k$.
Then the sum of the reciprocals of $\# \Aut_{\mathbb{F}_{11}}(C_i)$ for $1 \leq i \leq k$ is equal to $1$.
\end{cor}

\begin{proof}
It suffices to consider the following nine cases:
\begin{enumerate}
\item $C_1^{\rm (N1)}$, $C_6^{\rm (N1)}$, $C_7^{\rm (N1)}$, $C_1^{\rm (N2)}$, $C_3^{\rm (N2)}$ and $C_5^{\rm (N2)}$.
\item $C_2^{\rm (N1)}$ and $C_{4}^{\rm (N1)}$.
\item $C_3^{\rm (N1)}$, $C_5^{\rm (N1)}$, $C_8^{\rm (N1)}$, $C_2^{\rm (N2)}$ and $C_4^{\rm (N2)}$.
\item $C_1^{\rm (Dege)}$, $C_2^{\rm (Dege)}$, $C_9^{\rm (Dege)}$, $C_{10}^{\rm (Dege)}$, $C_{11}^{\rm (Dege)}$ and $C_{15}^{\rm (Dege)}$.
\item $C_3^{\rm (Dege)}$, $C_{12}^{\rm (Dege)}$, $C_{13}^{\rm (Dege)}$, $C_{16}^{\rm (Dege)}$ and $C_{17}^{\rm (Dege)}$.
\item $C_4^{\rm (Dege)}$ and $C_{5}^{\rm (Dege)}$.
\item $C_6^{\rm (Dege)}$.
\item $C_7^{\rm (Dege)}$ and $C_{14}^{\rm (Dege)}$.
\item $C_8^{\rm (Dege)}$.
\end{enumerate}
For the case (1), we have that the six curves $V(Q^{\rm (N1)},P_1^{\rm (N1)})$, $V(Q^{\rm (N1)},P_6^{\rm (N1)})$, $V(Q^{\rm (N1)},P_7^{\rm (N1)})$, $V(Q^{\rm (N2)},P_1^{\rm (N2)})$, $V(Q^{\rm (N2)},P_3^{\rm (N2)})$ and $V(Q^{\rm (N2)},P_5^{\rm (N2)})$ are isomorphic to each other over $\overline{\mathbb{F}_{11}}$.
By Theorems \ref{theo:N1Aut} and \ref{theo:N2Aut}, the orders of their automorphism groups over $\mathbb{F}_{11}$ are $6$, $4$, $12$, $12$, $4$ and $6$, respectively.
The sum of the reciprocals of the orders is
\[
\frac{1}{6}+\frac{1}{4}+\frac{1}{12}+\frac{1}{12}+\frac{1}{4}+\frac{1}{6}=1.
\]
Similarly to the case (1), we can prove the statements for the other cases (2) -- (9).
\end{proof}

\subsection{Automorphism groups over the algebraic closure $\overline{\F_{11}}$}\label{subsec:AutAC}
The nonhyperelliptic superspecial curves of genus $4$ over $\mathbb{F}_{11}$ have the $9$ isomorphism classes over the algebraic closure $\overline{\mathbb{F}_{11}}$.
In this subsection, we determine their automorphism groups over $\overline{\mathbb{F}_{11}}$.
Namely, we compute $\mathrm{Aut}(C):=\mathrm{Aut}_{\overline{\mathbb{F}_{11}}} (C)$ for the $\overline{\mathbb{F}_{11}}$-isomorphism class $C$ of each nonhyperelliptic superspecial curve of genus $4$ over $\mathbb{F}_{11}$.
Recall from Corollary \ref{cor:KH17overF11} that any nonhyperelliptic superspecial curve of genus $4$ over $\F_{11}$ is isomorphic over $\overline{\F_{11}}$ to one of the curves $C_i:=V(Q^{{\rm (N1)}}, P_i^{({\rm alc})})$ for $1 \leq i \leq 3$, or $C_i:=V(Q^{({\rm Dege})}, P_j^{({\rm alc})})$ for $4 \leq j \leq 9$, where $Q^{{\rm (N1)}} = 2 x w + 2 y z$, $Q^{({\rm Dege})} = 2 yw + z^2$ and

\begin{eqnarray}
P_1^{(\mathrm{alc})}&:=& x^2 y + x^2 z + x z^2 + 4 y^3 + 2 y^2 z + 10 y^2 w + 3 y z^2 + 8 y z w + 8 y w^2 \nonumber \\
& & + 8 z^3 + 7 z^2 w + 7 z w^2 + 4 w^3, \nonumber \\
P_2^{(\mathrm{alc})} &:=& x^2 y + x^2 z + y^3 + y^2 z + 7 y z^2 + 4 y w^2 + 2 z^3 + 9 z w^2, \nonumber \\
P_3^{(\mathrm{alc})} &:=& x^2 y + x^2 z + x z^2 + 9 y^3 + 6 y^2 z + 5 y^2 w + 8 y z^2 + 5 y z w + 2 y w^2 + z^3 + 2 z^2 w \nonumber \\
& & + 7 z w^2 + w^3, \nonumber \\
P_4^{(\mathrm{alc})} &:=&  x^3+y^3+8 y z^2 + z^2 w+ 2 w^3, \nonumber \\
P_5^{(\mathrm{alc})} &:=& x^3+2 y^3+4 y z^2 + z^2 w+10 w^3, \nonumber \\
P_6^{(\mathrm{alc})} &:=& x^3 + x w^2 + y^3, \nonumber \\
P_7^{(\mathrm{alc})} &:=& x^3 + x z w + y^3 + 7 z^3 + w^3, \nonumber \\
P_8^{(\mathrm{alc})} &:=& x^3 + x y z + x w^2 + y^3 + 5 z^3 + 4 w^3, \nonumber \\
P_9^{(\mathrm{alc})} &:=& x^3 + x y z + 6 x w^2 + y^3 + 8 z^3 + 8 w^3. \nonumber 
\end{eqnarray}

\begin{rem}
To compare the sizes of the orders of automorphism groups over $\overline{\mathbb{F}_{11}}$ with that over $\mathbb{F}_{11}$, we choose $C$ among elements in an $\overline{\mathbb{F}_{11}}$-isomorphism class so that $\mathrm{Aut}_{\mathbb{F}_{11}} (C)$ is maximal, i.e., $\# \mathrm{Aut}_{\mathbb{F}_{11}} (C) \geq \# \mathrm{Aut}_{\mathbb{F}_{11}} (C^{\prime})$ for any $C^{\prime}$ with $C^{\prime} \cong C$ over $\overline{\mathbb{F}_{11}}$.
For example, we choose $C_7^{{\rm (N1)}}$ among $C_1^{\rm (N1)}$, $C_6^{\rm (N1)}$, $C_7^{\rm (N1)}$, $C_1^{\rm (N2)}$, $C_3^{\rm (N2)}$ and $C_5^{\rm (N2)}$.
For this reason, some equations of $P_i^{{\rm (alc)}}$ are different from those given in Corollary \ref{cor:KH17overF11}.
\end{rem}

\begin{thm}\label{theo:AlcAut}
Put $C_i:=V(Q^{{\rm (N1)}}, P_i^{({\rm alc})})$ for $1 \leq i \leq 3$, and $C_{j}:=V(Q^{({\rm Dege})}, P_j^{({\rm alc})})$ for $4 \leq j \leq 9$, where $Q^{{\rm (N1)}} := 2 x w + 2 y z$ and $Q^{({\rm Dege})}:= 2 yw + z^2$.
Then we have the following isomorphisms:

\begin{tabular}{clcl}
$(1)$ & $\Aut ( C_1 ) \cong {\rm D}_6$,  & 
$(2)$ & $\Aut ( C_2 ) \cong {\rm C}_2 \times {\rm C}_2$, \\[2mm]
$(3)$ & $\Aut ( C_3 ) \cong {\rm S}_4 $, &
$(4)$ & $\Aut ( C_4 ) \cong {\rm D}_6 \times {\rm C}_3$, \\[2mm]
$(5)$ & $\Aut ( C_5 ) \cong {\rm S}_4 \times {\rm C}_3$, & 
$(6)$ & $\Aut ( C_6 ) \cong {\rm C}_{12}$, \\[2mm]
$(7)$ & $\Aut ( C_7 ) \cong {\rm C}_3$, &
$(8)$ & $\Aut ( C_8 ) \cong {\rm A}_4$, \\[2mm]
$(9)$ & $\Aut ( C_9 ) \cong {\rm C}_3$.
& 
& 
\end{tabular}
\end{thm}

\begin{proof}
Similarly to the proofs of Theorems \ref{theo:N1Aut} and \ref{theo:DegenerateAut}, the claims (1) -- (9) follow from Proposition \ref{prop:ACAut} in Subsection \ref{subsubsec:compAC}.
\end{proof}

In Table \ref{tab:ACAut}, we summarize the results in Theorem \ref{theo:AlcAut} together with the results in Theorems \ref{theo:N1Aut} -- \ref{theo:DegenerateAut}.
Here an $\mathbb{F}_{11}$-form of $C$ is a (nonhyperelliptic superspecial) curve $C^{\prime}$ over $\mathbb{F}_{11}$ such that $C \cong C^{\prime}$ over $\overline{\mathbb{F}_{11}}$, where $\mathrm{Aut}(C)$ is the automorphism group over $\overline{\mathbb{F}_{11}}$.
Each of $C_i^{\rm (N1)}$, $C_j^{\rm (N2)}$ and $C_k^{\rm (Dege)}$ is defined at the beginning of this section (Section \ref{section4}).

\renewcommand{\arraystretch}{1.3}
\begin{table}[H]
\begin{center}
\begin{tabular}{ccc|ccc} \hline
Superspecial curves $C$ & \multirow{2}{*}{$\Aut(C)$} & \multirow{2}{*}{$\# \Aut(C)$} & $\mathbb{F}_{11}$-forms $C^{\prime}$ & \multirow{2}{*}{$\Aut_{\mathbb{F}_{11}}(C^{\prime})$} & \multirow{2}{*}{$\# \Aut_{\mathbb{F}_{11}}(C^{\prime})$} \\ 
over $\overline{\mathbb{F}_{11}}$ & & &  of $C$ & & \\ \hline
\multirow{6}{*}{$C_1$}  & \multirow{6}{*}{${\rm D}_6$} & \multirow{6}{*}{$12$} & $C_7^{\rm (N1)}$ & ${\rm D}_6$ & $12$ \\
 & & & $C_1^{\rm (N1)}$ & ${\rm C}_6$ & $6$ \\
 & & & $C_6^{\rm (N1)}$ & ${\rm C}_2 \times {\rm C}_2$ & $4$ \\
 & & & $C_1^{\rm (N2)}$ & ${\rm D}_6$ & $12$ \\
 & & & $C_3^{\rm (N2)}$ & ${\rm C}_2 \times {\rm C}_2$ & $4$ \\
 & & & $C_5^{\rm (N2)}$ & ${\rm C}_6$ & $6$ \\ \hline
\multirow{2}{*}{$C_2$} & \multirow{2}{*}{${\rm C}_2 \times {\rm C}_2$} & \multirow{2}{*}{$4$} & $C_2^{\rm (N1)}$  & ${\rm C}_2$ & $2$\\
 & & & $C_4^{\rm (N1)}$ & ${\rm C}_2$ & $2$ \\ \hline
\multirow{5}{*}{$C_3$}  & \multirow{5}{*}{${\rm S}_4$} & \multirow{5}{*}{$24$} & $C_3^{\rm (N1)}$ & ${\rm D}_4$ & $8$ \\
 & & & $C_5^{\rm (N1)}$ & ${\rm C}_3$ & $3$ \\
 & & & $C_8^{\rm (N1)}$ & ${\rm S}_4$ & $24$ \\
 & & & $C_2^{\rm (N2)}$ & ${\rm C}_2 \times {\rm C}_2$ & $4$ \\
 & & & $C_4^{\rm (N2)}$ & ${\rm C}_4$ & $4$ \\ \hline
\multirow{6}{*}{$C_4$} & \multirow{6}{*}{${\rm D}_6 \times {\rm C}_3$} & \multirow{6}{*}{$36$} & $C_1^{\rm (Dege)}$  & ${\rm C}_2 \times {\rm C}_2$ & $4$\\
 & & & $C_2^{\rm (Dege)}$ & ${\rm C}_2 \times {\rm C}_2$ & $4$ \\
 & & & $C_9^{\rm (Dege)}$ & ${\rm C}_6$ & $6$ \\
 & & & $C_{10}^{\rm (Dege)}$ & ${\rm D}_6$ & $12$ \\
 & & & $C_{11}^{\rm (Dege)}$ & ${\rm D}_6$ & $12$ \\
 & & & $C_{15}^{\rm (Dege)}$ & ${\rm C}_6$ & $6$ \\ \hline
\multirow{5}{*}{$C_5$} & \multirow{5}{*}{${\rm S}_4 \times {\rm C}_3$} & \multirow{5}{*}{$72$}  & $C_{12}^{\rm (Dege)}$ & ${\rm S}_4$ & $24$ \\
 & & & $C_{3}^{\rm (Dege)}$ & ${\rm C}_2 \times {\rm C}_2$ & $4$ \\
 & & & $C_{13}^{\rm (Dege)}$ & ${\rm C}_4$ & $4$ \\
 & & & $C_{16}^{\rm (Dege)}$ & ${\rm C}_3$ & $3$ \\
 & & & $C_{17}^{\rm (Dege)}$ & ${\rm D}_4$ & $8$ \\ \hline
 \multirow{2}{*}{$C_6$}  & \multirow{2}{*}{${\rm C}_{12}$} & \multirow{2}{*}{$12$} & $C_4^{\rm (Dege)}$ & ${\rm C}_2$ & $2$  \\
 & & & $C_{5}^{\rm (Dege)}$ & ${\rm C}_2$ & $2$ \\ \hline
 $C_7$ & ${\rm C}_{3}$ & $3$ & $C_6^{\rm (Dege)}$  & $\{ e \}$ & $1$ \\ \hline
 \multirow{2}{*}{$C_8$}  & \multirow{2}{*}{${\rm A}_4$} & \multirow{2}{*}{$12$} & $C_7^{\rm (Dege)}$ & ${\rm C}_2$  & $2$ \\
 & & & $C_{14}^{\rm (Dege)}$ & ${\rm C}_2$ & $2$ \\ \hline
 $C_9$ & ${\rm C}_{3}$ & $3$  & $C_8^{\rm (Dege)}$ & $\{ e \}$ & $1$ \\ \hline
\end{tabular}
\vspace{-3mm}
\caption{The automorphism groups $\mathrm{Aut} (C_i):=\mathrm{Aut}_{\overline{\mathbb{F}_{11}}} (C_i)$ of $C_i:=V(Q^{{\rm (N1)}}, P_i^{({\rm alc})})$ for $1 \leq i \leq 3$, and $C_j:=V(Q^{({\rm Dege})}, P_j^{({\rm alc})})$ for $4 \leq j \leq 9$.
Here $Q^{{\rm (N1)}}$, $Q^{\rm (Dege)}$ and $P_i^{\rm (alc)}$ are defined at the beginning of this subsection (Subsection \ref{subsec:AutAC}).}\label{tab:ACAut}
\end{center}
\end{table}

\subsection{Computational parts in the proofs of main theorems}\label{subsec:comp}

In this subsection, we give computational results, which help proving our main theorems (Theorems \ref{theo:N1Aut} -- \ref{theo:DegenerateAut} and \ref{theo:AlcAut}).
Our computational results are obtained by executing algorithms given in Subsections \ref{subsec:solving} and \ref{subsec:compaut}.
We implemented the algorithms over Magma \cite{MagmaHP} in its 64-bit version.
The source codes and log files are available at the web page of the first author \cite{KudoHP}.

\subsubsection{Case of (N1) over the prime field $\mathbb{F}_{11}$}\label{subsubsec:compN1}

\begin{prop}\label{prop:N1Aut}
Let $C_i^{\rm (N1)} = V ( Q^{\rm (N1)}, P_i^{\rm (N1)})$ denote the superspecial curve of genus $4$ over $\mathbb{F}_{11}$ defined by $Q^{\rm (N1)}$ and $P_i^{\rm (N1)}$ for each $1 \leq i \leq 8$.
Then we have the following:
\begin{enumerate}
\item[$(1)$] The group $\Aut_{\mathbb{F}_{11}} ( C_1^{\rm (N1)} )$ has order $6$ and it is generated by
\[
\begin{pmatrix}
2&4&5&1\\
3&6&5&1\\
7&7&1&10\\
6&6&10&1
\end{pmatrix}.
\]
\item[$(2)$] The group $\Aut_{\mathbb{F}_{11}} ( C_2^{\rm (N1)} )$ has order $2$ and it is generated by $\diag (-1,1,1,-1)$.
\item[$(3)$] The group $\Aut_{\mathbb{F}_{11}} ( C_3^{\rm (N1)} )$ has order $8$ and it is generated by
\[
\begin{pmatrix}
5&6&6&6\\
5&6&5&5\\
5&5&6&5\\
6&6&6&5
\end{pmatrix},
\quad \text{and} \quad
{\diag}(-1,1,1,-1),
\]
whose orders are $4$ and $2$ respectively.
\item[$(4)$] The group $\Aut_{\mathbb{F}_{11}} ( C_4^{\rm (N1)} )$ has order $2$ and it is generated by
\[
\begin{pmatrix}
10&6&6&3\\
10&1&6&6\\
2&10&1&6\\
9&2&10&10
\end{pmatrix}.
\]
\item[$(5)$] The group $\Aut_{\mathbb{F}_{11}} ( C_5^{\rm (N1)} )$ has order $3$ and it is generated by
\[
\begin{pmatrix}
5&7&3&9\\
8&7&7&9\\
10&3&0&0\\
6&8&0&0
\end{pmatrix}.
\]
\item[$(6)$] The group $\Aut_{\mathbb{F}_{11}} ( C_6^{\rm (N1)} )$ has order $4$ and it is generated by
\[
\begin{pmatrix}
2&3&10&7\\
1&9&5&10\\
5&2&9&3\\
3&5&1&2
\end{pmatrix},
\quad \text{and} \quad
\begin{pmatrix}
1&5&6&3\\
0&0&1&6\\
0&1&0&5\\
0&0&0&1
\end{pmatrix},
\]
whose orders are $2$ and $2$ respectively.
\item[$(7)$] The group $\Aut_{\mathbb{F}_{11}} ( C_7^{\rm (N1)} )$ has order $12$ and it is generated by
\[
\begin{pmatrix}
6&3&5&3\\
1&6&1&5\\
5&3&6&3\\
1&5&1&6
\end{pmatrix},
\quad \text{and} \quad
\begin{pmatrix}
0&6&0&10\\
0&8&0&7\\
1&0&2&0\\
6&0&3&0
\end{pmatrix},
\]
whose orders are $2$ and $6$ respectively.
\item[$(8)$] The group $\Aut_{\mathbb{F}_{11}} ( C_8^{\rm (N1)} )$ has order $24$ and it is generated by
\[
\begin{pmatrix}
0&0&5&2\\
7&6&9&8\\
0&0&8&9\\
2&6&1&8
\end{pmatrix},
\quad \text{and} \quad
\begin{pmatrix}
9&5&2&5\\
1&3&3&2\\
8&5&3&5\\
4&8&1&9
\end{pmatrix},
\]
whose orders are $4$ and $2$ respectively.
\end{enumerate}
\end{prop}

\begin{proof}
We prove only the statement (1) since the other cases (2) -- (8) are proved in ways similar to (1).
Let $\varphi$ denote the coefficient matrix associated to $Q:=Q^{\rm (N1)}$.
In the following, we compute $G_1:= \{\, g \in \tilde{\gO}_{\varphi}(\mathbb{F}_{11}) \mid g \cdot P_1^{\rm (N1)} \equiv r P_1^{\rm (N1)} \bmod{Q} \mbox{ for some } r \in \mathbb{F}_{11}^{\times} \,\}$, and then find a generator of $G_1 / {\sim} \cong \Aut_{\mathbb{F}_{11}} ( C_1^{\rm (N1)} )$, where $g \sim cg$ for $g \in \tilde{\gO}_{\varphi}(\mathbb{F}_{11})$ and $c \in \mathbb{F}_{11}^{\times}$.
Recall from Subsection \ref{subsec:orthoN1} that we have the Bruhat decomposition of the orthogonal similitude group associated to $Q$, say
\begin{eqnarray}
 \tilde\gO_{\varphi}( \mathbb{F}_{11})= \gA \tilde \gT \gU \gW \gU ,\label{eq:BruhatN1_pf}
\end{eqnarray}
where $\gA = \{ 1_4, M_{\rm A} \}$, $\tilde\gT$, $\gU$ and $\gW = \{ 1_4, s_1, s_2, s_1 s_2 \}$ are the same as in Subsection \ref{subsec:orthoN1}.
Putting
\begin{equation*}
\begin{array}{llll}
\mathrm{\Omega}_1 := \tilde{\gT} \gU \gU, & \mathrm{\Omega}_2 := \tilde{\gT} \gU s_1 \gU, & \mathrm{\Omega}_3 := \tilde{\gT} \gU s_2 \gU, & \mathrm{\Omega}_4 := \tilde{\gT} \gU s_1 s_2 \gU, \\
\mathrm{\Omega}_5 := M_{\gA} \tilde{\gT} \gU \gU, & \mathrm{\Omega}_6 := M_{\gA} \tilde{\gT} \gU s_1 \gU, & \mathrm{\Omega}_7 := M_{\gA} \tilde{\gT} \gU s_2 \gU, & \mathrm{\Omega}_8 := M_{\gA} \tilde{\gT} \gU s_1 s_2 \gU ,
\end{array}
\end{equation*}
we have $ \tilde\gO_{\varphi}( \mathbb{F}_{11}) = \bigcup_{i=1}^8 \Omega_i$.
Using the same notation as in Subsection \ref{subsec:compaut}, we set \begin{eqnarray}
g_1 &:=& \tilde{T} (a_1,a_2, b_1,b_2,c) U (d_1, d_2 ), \nonumber \\
g_2 &:=& \tilde{T} (a_1,a_2, b_1,b_2,c) U (d_1, d_2) s_1 U_1 (e_1), \nonumber \\
g_3 &:=& \tilde{T} (a_1,a_2, b_1,b_2,c) U (d_1, d_2) s_2 U_2 (e_2), \nonumber \\
g_4 &:=& \tilde{T} (a_1,a_2, b_1,b_2,c) U (d_1, d_2) s_1 s_2 U (e_1, e_2 ), \nonumber \\
g_5 &:=& M_{\gA} \tilde{T} (a_1,a_2, b_1,b_2,c) U (d_1, d_2 ), \nonumber \\
g_6 &:=& M_{\gA} \tilde{T} (a_1,a_2, b_1,b_2,c) U (d_1, d_2) s_1 U_1 (e_1), \nonumber \\
g_7 &:=& M_{\gA} \tilde{T} (a_1,a_2, b_1,b_2,c) U (d_1, d_2) s_2 U_2 (e_2), \nonumber \\
g_8 &:=& M_{\gA} \tilde{T} (a_1,a_2, b_1,b_2,c) U (d_1, d_2) s_1 s_2 U (e_1, e_2 ),\nonumber 
\end{eqnarray}
where $a_1$, $a_2$, $b_1$, $b_2$, $c$, $d_1$, $d_2$, $e_1$ and $e_2$ are variables.
For the inputs $P_1^{\rm (N1)}$, $q=11$ and $q^{\prime}=q$, we execute {\rm Main algorithm} (its pseudocode is given in Algorithm \ref{alg:AutoN1}) in Subsection \ref{subsec:compaut}.
In Main algorithm, we set $\mathcal{G}_Q = \{ g_i \mid 1 \leq i \leq 8 \}$.
For solving multivariate systems in the algorithm, we use Algorithm \ref{alg:VarietyAC1} as a subroutine, and adopt the grevlex order with
\[
a_1 \succ a_2 \succ b_1 \succ b_2 \succ c \succ s \succ d_1 \succ d_2 \succ e_1 \succ e_2 \succ r.
\]
Moreover, we set $\mathcal{P}_Q:=\{ a_1 a_2 - 1, b_1 b_2 - 1, c s -1, r-1 \}$ in Main algorithm, where $s$ and $r$ are extra variables.
Namely, we add $a_1 a_2=1$, $b_1 b_2=1$, $c s = 1$ and $r=1$ to each multivariate system.
Note that we may assume $r \equiv 1$ since the $3$-rd power map over $\mathbb{F}_{11}$ is surjective.
From the output of our computation, we obtain all elements $g$ in $G_1 = \{\, g \in \tilde{\gO}_{\varphi}(\mathbb{F}_{11}) \mid g \cdot P_1^{\rm (N1)} \equiv r P_1^{\rm (N1)} \bmod{Q} \mbox{ for some } r \in \mathbb{F}_{11}^{\times} \,\}$.
The set of the computed elements $g$ includes the matrix
\[
\begin{pmatrix}
2&4&5&1\\
3&6&5&1\\
7&7&1&10\\
6&6&10&1
\end{pmatrix},
\]
whose order is $6$.
It is computationally checked that the above matrix generates $G_1 / {\sim}$.
\end{proof}

\subsubsection{Case of (N2) over the prime field $\mathbb{F}_{11}$}\label{subsubsec:compN2}

\begin{prop}\label{prop:N2Aut}
Let $C_i^{\rm (N2)} = V ( Q^{\rm (N2)}, P_i^{\rm (N2)})$ denote the superspecial curve of genus $4$ over $\mathbb{F}_{11}$ defined by $Q^{\rm (N2)}$ and $P_i^{\rm (N2)}$ for each $1 \leq i \leq 5$.
Then we have the following:
\begin{enumerate}
\item[$(1)$] The group $\Aut_{\mathbb{F}_{11}} ( C_1^{\rm (N2)} )$ has order $12$ and it is generated by
\[
\begin{pmatrix}
1&6&9&5\\
10&8&6&6\\
1&8&1&1\\
6&10&9&1
\end{pmatrix},
\quad \text{and} \quad
\begin{pmatrix}
1&0&8&5\\
10&1&3&6\\
1&0&7&1\\
6&1&2&1
\end{pmatrix},
\]
whose orders are $2$ and $6$ respectively.
\item[$(2)$] The group $\Aut_{\mathbb{F}_{11}} ( C_2^{\rm (N2)} )$ has order $4$ and it is generated by
\[
{\diag}(-1,1,1,-1),
\quad \text{and} \quad
\begin{pmatrix}
0&0&0&9\\
0&1&0&0\\
0&0&1&0\\
5&0&0&0
\end{pmatrix},
\]
whose orders are $2$ and $2$ respectively.
\item[$(3)$] The group $\Aut_{\mathbb{F}_{11}} ( C_3^{\rm (N2)} )$ has order $4$ and it is generated by
\[
{\diag}(-1,1,1,-1),
\quad \text{and} \quad
\begin{pmatrix}
0&0&0&9\\
0&1&0&0\\
0&0&1&0\\
5&0&0&0
\end{pmatrix},
\]
whose orders are $2$ and $2$ respectively.
\item[$(4)$] The group $\Aut_{\mathbb{F}_{11}} ( C_4^{\rm (N2)} )$ has order $4$ and it is generated by
\[
\begin{pmatrix}
10&8&5&1\\
7&4&7&5\\
5&5&3&10\\
5&7&0&5
\end{pmatrix}.
\]
\item[$(5)$] The group $\Aut_{\mathbb{F}_{11}} ( C_5^{\rm (N2)} )$ has order $6$ and it is generated by
\[
\begin{pmatrix}
5&2&2&2\\
10&7&8&8\\
10&2&1&2\\
10&1&6&8
\end{pmatrix}.
\]
\end{enumerate}
\end{prop}

\begin{proof}
We prove only the statement (1) since the other cases (2) -- (5) are proved in ways similar to (1).
Let $\varphi$ denote the coefficient matrix associated to $Q:=Q^{\rm (N2)}$.
In the following, we compute $G_1:= \{\, g \in \tilde{\gO}_{\varphi}(\mathbb{F}_{11}) \mid g \cdot P_1^{\rm (N2)} \equiv r P_1^{\rm (N2)} \bmod{Q} \mbox{ for some } r \in \mathbb{F}_{11}^{\times} \,\}$, and then find generators of $G_1 / {\sim} \cong \Aut_{\mathbb{F}_{11}} ( C_1^{\rm (N2)} )$, where $g \sim cg$ for $g \in \tilde{\gO}_{\varphi}(\mathbb{F}_{11})$ and $c \in \mathbb{F}_{11}^{\times}$.
Recall from Subsection \ref{subsec:orthoN2} that we have the Bruhat decomposition of the orthogonal similitude group associated to $Q$, say
\begin{eqnarray}
 \tilde\gO_{\varphi}( \mathbb{F}_{11})= \gA \tilde{\gT} \gU \gW \gU ,\label{eq:BruhatN2_pf}
\end{eqnarray}
where $\gA = \{ 1_4, M_{\rm A} \}$, $\tilde{\gT}$, $\gU$ and $\gW = \{ 1_4, M_{\rm W} \}$ are the same as in Subsection \ref{subsec:orthoN2}.
Putting
\begin{equation*}
\begin{array}{llll}
\mathrm{\Omega}_1 := \tilde{\rm{T}} {\rm U} {\rm U}, & {\rm \Omega}_2 := \tilde{{\rm T}} {\rm U} M_{\rm W} {\rm U}, & {\rm \Omega}_3 := M_{\rm A} \tilde{{\rm T}} {\rm U} {\rm U}, & {\rm \Omega}_4 := M_{\rm A} \tilde{{\rm T}} {\rm U} M_{\rm W} {\rm U},
\end{array}
\end{equation*}
we have $ \tilde\gO_{\varphi}( \mathbb{F}_{11}) = \bigcup_{i=1}^4 \Omega_i$.
Using the same notation as in Subsection \ref{subsec:compaut}, we set \begin{eqnarray}
g_1 &:=& \tilde{T}(a_1,a_2,b_1,b_2) U (c_1, c_2 ), \nonumber \\
g_2 &:=& \tilde{T}(a_1,a_2,b_1,b_2) U (c_1, c_2 ) M_{\gW} U (d_1, d_2 ), \nonumber \\
g_3 &:=& M_{\gA} \tilde{T}(a_1,a_2,b_1,b_2) U (c_1, c_2 ), \nonumber \\
g_4 &:=& M_{\gA} \tilde{T}(a_1,a_2,b_1,b_2) U (c_1, c_2 ) M_{\gW} U (d_1, d_2 ), \nonumber 
\end{eqnarray}
where $a_1$, $a_2$, $b_1$, $b_2$, $c_1$, $c_2$, $d_1$ and $d_2$ are variables.
For the inputs $P_1^{\rm (N2)}$, $q=11$ and $q^{\prime}=q$, we execute {\rm Main algorithm} (its pseudocode is given in Algorithm \ref{alg:AutoN1}) in Subsection \ref{subsec:compaut}.
In Main algorithm, we set $\mathcal{G}_Q = \{ g_i \mid 1 \leq i \leq 4 \}$.
For solving multivariate systems in the algorithm, we use Algorithm \ref{alg:VarietyAC1} as a subroutine, and adopt the grevlex order with
\[
 a_1 \succ a_2 \succ b_1 \succ b_2 \succ c_1 \succ c_2 \succ d_1 \succ d_2 \succ t \succ r. 
\]
Moreover, we set $\mathcal{P}_Q:=\{ a_1 a_2 - 1, (b_1^2-\epsilon b_2^2) t -1, r-1 \}$ in Main algorithm, where $t$ and $r$ are extra variables.
Namely, we add $a_1 a_2=1$, $(b_1^2-\epsilon b_2^2) t =1$ and $r=1$ to each multivariate system.
Note that we may assume $r \equiv 1$ since the $3$-rd power map over $\mathbb{F}_{11}$ is surjective.
From the output of our computation, we obtain all elements $g$ in $G_1 = \{\, g \in \tilde{\gO}_{\varphi}(\mathbb{F}_{11}) \mid g \cdot P_1^{\rm (N2)} \equiv r P_1^{\rm (N2)} \bmod{Q} \mbox{ for some } r \in \mathbb{F}_{11}^{\times} \,\}$.
The set of the computed elements $g$ includes the matrices
\[
\begin{pmatrix}
1&6&9&5\\
10&8&6&6\\
1&8&1&1\\
6&10&9&1
\end{pmatrix},
\quad \text{and} \quad
\begin{pmatrix}
1&0&8&5\\
10&1&3&6\\
1&0&7&1\\
6&1&2&1
\end{pmatrix},
\]
whose orders are $2$ and $6$ respectively.
It is computationally checked that these two matrices generate $G_1 / {\sim}$.
\end{proof}

\subsubsection{Case of (Dege) over the prime field $\mathbb{F}_{11}$}\label{subsubsec:compDege}

\begin{prop}\label{prop:DegeAut}
Let $C_i^{\rm (Dege)} = V ( Q^{\rm (Dege)}, P_i^{\rm (Dege)})$ denote the superspecial curve of genus $4$ over $\mathbb{F}_{11}$ defined by $Q^{\rm (Dege)}$ and $P_i^{\rm (Dege)}$ for each $1 \leq i \leq 17$.
Then we have the following:
\begin{enumerate}
\item[$(1)$] The group $\Aut_{\mathbb{F}_{11}} ( C_1^{\rm (Dege)} )$ has order $4$ and it is generated by
\[
{\diag}(1,1,-1,1),
\quad \text{and} \quad
\begin{pmatrix}
1&0&0&0\\
0&0&0&1\\
0&0&1&0\\
0&1&0&0
\end{pmatrix},
\]
whose orders are $2$ and $2$ respectively.
\item[$(2)$] The group $\Aut_{\mathbb{F}_{11}} ( C_2^{\rm (Dege)} )$ has order $4$ and it is generated by
\[
{\diag}(1,1,-1,1),
\quad \text{and} \quad
\begin{pmatrix}
1&0&0&0\\
0&0&0&7\\
0&0&1&0\\
0&8&0&0
\end{pmatrix},
\]
whose orders are $2$ and $2$ respectively.
\item[$(3)$] The group $\Aut_{\mathbb{F}_{11}} ( C_3^{\rm (Dege)} )$ has order $4$ and it is generated by
\[
\begin{pmatrix}
1&0&0&0\\
0&0&0&3\\
0&0&1&0\\
0&4&0&0
\end{pmatrix},
\quad \text{and} \quad
\begin{pmatrix}
1&0&0&0\\
0&7&4&2\\
0&5&7&4\\
0&10&5&7
\end{pmatrix},
\]
whose orders are $2$ and $2$ respectively.
\item[$(4)$] The group $\Aut_{\mathbb{F}_{11}} ( C_4^{\rm (Dege)} )$ has order $2$ and it is generated by ${\diag}(1,1,-1,1)$.
\item[$(5)$] The group $\Aut_{\mathbb{F}_{11}} ( C_5^{\rm (Dege)} )$ has order $2$ and it is generated by ${\diag}(1,1,-1,1)$.
\item[$(6)$] The group $\Aut_{\mathbb{F}_{11}} ( C_6^{\rm (Dege)} )$ has order $1$.
\item[$(7)$] The group $\Aut_{\mathbb{F}_{11}} ( C_7^{\rm (Dege)} )$ has order $2$ and it is generated by
\[
\begin{pmatrix}
1&0&0&0\\
0&7&1&7\\
0&9&7&1\\
0&6&9&7
\end{pmatrix}.
\]
\item[$(8)$] The group $\Aut_{\mathbb{F}_{11}} ( C_8^{\rm (Dege)} )$ has order $1$.
\item[$(9)$] The group $\Aut_{\mathbb{F}_{11}} ( C_9^{\rm (Dege)} )$ has order $6$ and it is generated by
\[
\begin{pmatrix}
1&0&0&0\\
0&6&5&8\\
0&3&7&3\\
0&2&4&7
\end{pmatrix}.
\]
\item[$(10)$] The group $\Aut_{\mathbb{F}_{11}} ( C_{10}^{\rm (Dege)} )$ has order $12$ and it is generated by
\[
\begin{pmatrix}
1&0&0&0\\
0&2&9&10\\
0&6&6&9\\
0&2&6&2
\end{pmatrix},
\quad \text{and} \quad
\begin{pmatrix}
1&0&0&0\\
0&2&2&10\\
0&6&5&9\\
0&2&5&2
\end{pmatrix},
\]
whose orders are $2$ and $6$ respectively.
\item[$(11)$] The group $\Aut_{\mathbb{F}_{11}} ( C_{11}^{\rm (Dege)} )$ has order $12$ and it is generated by
\[
\begin{pmatrix}
1&0&0&0\\
0&7&6&10\\
0&8&9&6\\
0&8&8&7
\end{pmatrix},
\quad \text{and} \quad
\begin{pmatrix}
1&0&0&0\\
0&2&1&8\\
0&2&0&3\\
0&10&6&7
\end{pmatrix},
\]
whose orders are $2$ and $6$ respectively.
\item[$(12)$] The group $\Aut_{\mathbb{F}_{11}} ( C_{12}^{\rm (Dege)} )$ has order $24$ and it is generated by
\[
\begin{pmatrix}
1&0&0&0\\
0&1&9&9\\
0&8&5&4\\
0&1&6&4
\end{pmatrix},
\quad \text{and} \quad
\begin{pmatrix}
1&0&0&0\\
0&0&0&3\\
0&0&1&0\\
0&4&0&0
\end{pmatrix},
\]
whose orders are $4$ and $2$ respectively.
\item[$(13)$] The group $\Aut_{\mathbb{F}_{11}} ( C_{13}^{\rm (Dege)} )$ has order $4$ and it is generated by
\[
\begin{pmatrix}
1&0&0&0\\
0&9&7&4\\
0&1&1&0\\
0&3&0&0
\end{pmatrix}.
\]
\item[$(14)$] The group $\Aut_{\mathbb{F}_{11}} ( C_{14}^{\rm (Dege)} )$ has order $2$ and it is generated by
\[
\begin{pmatrix}
1&0&0&0\\
0&1&4&3\\
0&10&8&4\\
0&5&10&1
\end{pmatrix}.
\]
\item[$(15)$] The group $\Aut_{\mathbb{F}_{11}} ( C_{15}^{\rm (Dege)} )$ has order $6$ and it is generated by
\[
\begin{pmatrix}
1&0&0&0\\
0&8&7&10\\
0&5&2&3\\
0&6&10&10
\end{pmatrix}.
\]
\item[$(16)$] The group $\Aut_{\mathbb{F}_{11}} ( C_{16}^{\rm (Dege)} )$ has order $3$ and it is generated by
\[
\begin{pmatrix}
1&0&0&0\\
0&1&8&1\\
0&8&10&0\\
0&1&0&0
\end{pmatrix}.
\]
\item[$(17)$] The group $\Aut_{\mathbb{F}_{11}} ( C_{17}^{\rm (Dege)} )$ has order $8$ and it is generated by
\[
\begin{pmatrix}
1&0&0&0\\
0&3&6&5\\
0&2&3&9\\
0&3&3&4
\end{pmatrix},
\quad \text{and} \quad
\begin{pmatrix}
1&0&0&0\\
0&7&2&6\\
0&2&9&2\\
0&6&2&7
\end{pmatrix},
\]
whose orders are $4$ and $2$ respectively.
\end{enumerate}
\end{prop}

\begin{proof}
We prove only the statement (1) since the other cases (2) -- (17) are proved in ways similar to (1).
Let $\varphi$ denote the coefficient matrix associated to $Q:=Q^{\rm (Dege)}$.
In the following, we compute $G_1:= \{\, g \in \tilde{\gO}_{\varphi}(\mathbb{F}_{11}) \mid g \cdot P_1^{\rm (Dege)} \equiv r P_1^{\rm (Dege)} \bmod{Q} \mbox{ for some } r \in \mathbb{F}_{11}^{\times} \,\}$, and then find generators of $G_1 / {\sim} \cong \Aut_{\mathbb{F}_{11}} ( C_1^{\rm (Dege)} )$, where $g \sim cg$ for $g \in \tilde{\gO}_{\varphi}(\mathbb{F}_{11})$ and $c \in \mathbb{F}_{11}^{\times}$.
Recall from Subsection \ref{subsec:orthoDege} that we have the Bruhat decomposition of the orthogonal similitude group associated to $Q$, say
\begin{eqnarray}
 \tilde{\gO}_{\varphi}( \mathbb{F}_{11})= ({\rm A} \tilde{\rm T} {\rm U} \sqcup {\rm A} \tilde{\rm T} {\rm U} M_{\rm W} \gU) \gV, \label{eq:BruhatDege_pf}
\end{eqnarray}
where $\gA  = \{ 1_4, M_{\rm A} \}$, $\tilde{\gT}$, $M_{\rm W}$, $\gU$ and $\gV$ are the same as in Subsection \ref{subsec:orthoDege}.
Putting
\begin{equation*}
\begin{array}{llll}
{\rm \Omega}_1 := \tilde{\rm T} \gU \gV, & {\rm \Omega}_2 := \tilde{\rm T} {\rm U} M_{\rm W} {\rm U} {\rm V}, & {\rm \Omega}_3 := M_{\rm A} \tilde{\rm T} {\rm U} {\rm V}, & {\rm \Omega}_4 := M_{\rm A} \tilde{\rm T} {\rm U} M_{\rm W} {\rm U} {\rm V},
\end{array}
\end{equation*}
we have $ \tilde\gO_{\varphi}( \mathbb{F}_{11}) = \bigcup_{i=1}^4 \Omega_i$.
Using the same notation as in Subsection \ref{subsec:compaut}, we set \begin{eqnarray}
g_1 &:=& \tilde{T} (a_1,a_2, a_3) U (b) V (d, e_1, e_2, e_3) , \nonumber \\
g_2 &:=& \tilde{T} (a_1,a_2, a_3) U (b) M_{\rm W} U (c) V (d, e_1, e_2, e_3), \nonumber \\
g_3 &:=& M_{\rm A} \tilde{T} (a_1,a_2, a_3) U (b) V (d, e_1, e_2, e_3), \nonumber \\
g_4 &:=& M_{\rm A} \tilde{T} (a_1,a_2, a_3) U (b) M_{\rm W} U (c) V (d, e_1, e_2, e_3), \nonumber
\end{eqnarray}
where $a_1$, $a_2$, $a_3$, $b$, $c$, $d$, $e_1$, $e_2$ and $e_3$ are variables.
For the inputs $P_1^{\rm (Dege)}$, $q=11$ and $q^{\prime}=q$, we execute {\rm Main algorithm} (its pseudocode is given in Algorithm \ref{alg:AutoN1}) in Subsection \ref{subsec:compaut}.
In Main algorithm, we set $\mathcal{G}_Q = \{ g_i \mid 1 \leq i \leq 4 \}$.
For solving multivariate systems in the algorithm, we use Algorithm \ref{alg:VarietyAC1} as a subroutine, and adopt the grevlex order with
\[
a_1 \succ a_2 \succ a_3 \succ s \succ b \succ c \succ d \succ t \succ e_1 \succ e_2 \succ e_3 \succ r.
\]
Moreover, we set $\mathcal{P}_Q:=\{ a_1 a_2 - 1, a_3 s - 1, d t - 1, r-1 \}$ in Main algorithm, where $s$, $t$ and $r$ are extra variables.
Namely, we add $a_1 a_2=1$, $a_3 s = 1$, $d t=1$ and $r=1$ to each multivariate system.
Note that we may assume $r \equiv 1$ since the $3$-rd power map over $\mathbb{F}_{11}$ is surjective.
From the output of our computation, we obtain all elements $g$ in $G_1 = \{\, g \in \tilde{\gO}_{\varphi}(\mathbb{F}_{11}) \mid g \cdot P_1^{\rm (Dege)} \equiv r P_1^{\rm (Dege)} \bmod{Q} \mbox{ for some } r \in \mathbb{F}_{11}^{\times} \,\}$.
The set of the computed elements $g$ includes the matrices
\[
{\diag}(1,1,-1,1),
\quad \text{and} \quad
\begin{pmatrix}
1&0&0&0\\
0&0&0&1\\
0&0&1&0\\
0&1&0&0
\end{pmatrix},
\]
whose orders are $2$ and $2$ respectively.
It is computationally checked that these two matrices generate $G_1 / {\sim}$.
\end{proof}

\subsubsection{Computational results over the algebraic closure $\overline{\mathbb{F}_{11}}$}\label{subsubsec:compAC}

\begin{prop}\label{prop:ACAut}
Let $Q^{\rm (N1)}= 2 x w + 2 y z$, and $Q^{\rm (Dege)}= 2 y w + z^2$ in $\mathbb{F}_{11}[x,y,z,w]$.
For each $1 \leq i \leq 3$, we denote by $C_i$ the superspecial curve $V ( Q^{\rm (N1)}, P_i^{\rm (alc)})$ over $\overline{\mathbb{F}_{11}}$ defined by $Q^{\rm (N1)}$ and $P_i^{\rm (alc)}$.
For each $4 \leq j \leq 9$, we denote by $C_j $ the superspecial curve $V ( Q^{\rm (Dege)}, P_j^{\rm (alc)})$ over $\overline{\mathbb{F}_{11}}$ defined by $Q^{\rm (Dege)}$ and $P_j^{\rm (alc)}$.
Then we have the following:
\begin{enumerate}
\item[$(1)$] The group $\Aut ( C_1 )$ has order $12$ and it is generated by
\[
\begin{pmatrix}
6&3&5&3\\
1&6&1&5\\
5&3&6&3\\
1&5&1&6
\end{pmatrix},
\quad \text{and} \quad
\begin{pmatrix}
0&6&0&10\\
0&8&0&7\\
1&0&2&0\\
6&0&3&0
\end{pmatrix},
\]
whose orders are $2$ and $6$ respectively.
\item[$(2)$] The group $\Aut ( C_2 )$ has order $4$ and it is generated by
\[
{\diag}(-1,1,1,-1),
\quad \text{and} \quad
\begin{pmatrix}
0&0&0&\zeta^{-6} \\
0&0&7&0\\
0&8&0&0\\
\zeta^{6}&0&0&0
\end{pmatrix},
\]
whose orders are $2$ and $2$ respectively.
\item[$(3)$] The group $\Aut ( C_3 )$ has order $24$ and it is generated by
\[
\begin{pmatrix}
0&0&5&2\\
7&6&9&8\\
0&0&8&9\\
2&6&1&8
\end{pmatrix},
\quad \text{and} \quad
\begin{pmatrix}
9&5&2&5 \\
1&3&3&2\\
8&5&3&5\\
4&8&1&9
\end{pmatrix},
\]
whose orders are $4$ and $2$ respectively.
\item[$(4)$] The group $\Aut (C_4 )$ has order $36$ and it is generated by
\[
\begin{pmatrix}
1&0&0&0\\
0&2&2&10\\
0&6&5&9\\
0&2&5&2
\end{pmatrix},
\quad
{\diag}(1,\zeta^{80},\zeta^{80},\zeta^{80}),
\quad \text{and} \quad
\begin{pmatrix}
1&0&0&0\\
0&2&9&10\\
0&6&6&9\\
0&2&6&2
\end{pmatrix},
\]
whose orders are $6$, $3$ and $2$ respectively.
\item[$(5)$] The group $\Aut ( C_5 )$ has order $72$ and it is generated by
\[
\begin{pmatrix}
1&0&0&0\\
0&5&2&4\\
0&8&0&9\\
0&9&3&5
\end{pmatrix},
\quad
\begin{pmatrix}
1&0&0&0\\
0&3&9&3\\
0&6&6&9\\
0&5&6&3
\end{pmatrix},
\quad \text{and} \quad
{\diag}(1,\zeta^{80},\zeta^{80},\zeta^{80}),
\]
whose orders are $4$, $2$ and $3$ respectively.
\item[$(6)$] The group $\Aut ( C_6 )$ has order $12$ and it is generated by ${\diag}(1,\zeta^{80},\zeta^{70},-1)$.
\item[$(7)$] The group $\Aut ( C_7 )$ has order $3$ and it is generated by ${\diag}(1,1,\zeta^{-40},\zeta^{40})$.
\item[$(8)$] The group $\Aut ( C_8 )$ has order $12$ and it is generated by
\[
\begin{pmatrix}
1&0&0&0\\
0&\zeta^{4}&\zeta^{-40}&7\\
0&\zeta^{-8}&\zeta^{44}&1\\
0&\zeta^{28}&\zeta^{32}&7
\end{pmatrix},
\quad \text{and} \quad
\begin{pmatrix}
1&0&0&0 \\
0&7&\zeta^{40}&\zeta^{44}\\
0&\zeta^{32}&7&\zeta^{40}\\
0&\zeta^{28}&\zeta^{32}&7
\end{pmatrix},
\]
whose orders are $3$ and $2$ respectively.
\item[$(9)$] The group $\Aut ( C_9 )$ has order $3$ and it is generated by ${\diag}(1,\zeta^{40},\zeta^{-40},1)$.
\end{enumerate}
Here $\zeta$ is a root of $a^2 + 7 a + 2$, which is a primitive element of $\mathbb{F}_{121}$.
\end{prop}

\begin{proof}
We prove only the statement (1) since the other cases (2) -- (9) are proved in ways similar to (1).
Let $\varphi$ denote the coefficient matrix associated to $Q:=Q^{\rm (N1)}$.
In the following, we compute $G_1:= \{\, g \in \tilde{\gO}_{\varphi}(\overline{\mathbb{F}_{11}}) \mid g \cdot P_1^{\rm (alc)} \equiv r P_1^{\rm (alc)} \bmod{Q} \mbox{ for some } r \in \overline{\mathbb{F}_{11}}^{\times} \,\}$, and then find generators of $G_1 / {\sim} \cong \Aut ( C_1^{\rm (alc)} )$, where $g \sim cg$ for $g \in \tilde{\gO}_{\varphi}(\overline{\mathbb{F}_{11}})$ and $c \in \overline{\mathbb{F}_{11}}^{\times}$.
Recall from Subsection \ref{subsec:orthoN1} that we have the Bruhat decomposition of the orthogonal similitude group associated to $Q$, say
\begin{eqnarray}
 \tilde\gO_{\varphi}( \overline{\mathbb{F}_{11}})= \gA \tilde \gT \gU \gW \gU ,\label{eq:BruhatN1_pf2}
\end{eqnarray}
where $\gA = \{ 1_4, M_{\rm A} \}$, $\tilde\gT$, $\gU$ and $\gW = \{ 1_4, s_1, s_2, s_1 s_2 \}$ are the same as in Subsection \ref{subsec:orthoN1}.
Putting
\begin{equation*}
\begin{array}{llll}
\Omega_1 := \tilde{\gT} \gU \gU, & \Omega_2 := \tilde{\gT} \gU s_1 \gU, & \Omega_3 := \tilde{\gT} \gU s_2 \gU, & \Omega_4 := \tilde{\gT} \gU s_1 s_2 \gU, \\
\Omega_5 := M_{\gA} \tilde{\gT} \gU \gU, & \Omega_6 := M_{\gA} \tilde{\gT} \gU s_1 \gU, & \Omega_7 := M_{\gA} \tilde{\gT} \gU s_2 \gU, & \Omega_8 := M_{\gA} \tilde{\gT} \gU s_1 s_2 \gU ,
\end{array}
\end{equation*}
we have $ \tilde\gO_{\varphi}( \overline{\mathbb{F}_{11}}) = \bigcup_{i=1}^8 \Omega_i$.
Using the same notation as in Subsection \ref{subsec:compaut}, we set \begin{eqnarray}
g_1 &:=& \tilde{T} (a_1,a_2, b_1,b_2,c) U (d_1, d_2 ), \nonumber \\
g_2 &:=& \tilde{T} (a_1,a_2, b_1,b_2,c) U (d_1, d_2) s_1 U_1 (e_1 ), \nonumber \\
g_3 &:=& \tilde{T} (a_1,a_2, b_1,b_2,c) U (d_1, d_2) s_2 U_2 (e_2), \nonumber \\
g_4 &:=& \tilde{T} (a_1,a_2, b_1,b_2,c) U (d_1, d_2) s_1 s_2 U (e_1, e_2 ), \nonumber \\
g_5 &:=& M_{\gA} \tilde{T} (a_1,a_2, b_1,b_2,c) U (d_1, d_2 ), \nonumber \\
g_6 &:=& M_{\gA} \tilde{T} (a_1,a_2, b_1,b_2,c) U (d_1, d_2) s_1 U_1 (e_1), \nonumber \\
g_7 &:=& M_{\gA} \tilde{T} (a_1,a_2, b_1,b_2,c) U (d_1, d_2) s_2 U_2 (e_2), \nonumber \\
g_8 &:=& M_{\gA} \tilde{T} (a_1,a_2, b_1,b_2,c) U (d_1, d_2) s_1 s_2 U (e_1, e_2 ). \nonumber 
\end{eqnarray}
For the inputs $P_1^{\rm (alc)}$, $q=11$ and $q^{\prime}=0$, we execute {\rm Main algorithm} (its pseudocode is given in Algorithm \ref{alg:AutoN1}) in Subsection \ref{subsec:compaut}.
In Main algorithm, we set $\mathcal{G}_Q = \{ g_i \mid 1 \leq i \leq 8 \}$.
For solving multivariate systems in the algorithm, we use Algorithm \ref{alg:VarietyAC1} as a subroutine, and adopt the grevlex order with
\[
a_1 \succ a_2 \succ b_1 \succ b_2 \succ c \succ s \succ d_1 \succ d_2 \succ e_1 \succ e_2 \succ r.
\]
Moreover, we set $\mathcal{P}_Q:=\{ a_1 a_2 - 1, b_1 b_2 - 1, c s - 1, r-1 \}$ in Main algorithm, where $s$ and $r$ are extra variables.
Namely, we add $a_1 a_2=1$, $b_1 b_2=1$, $c s = 1$ and $r=1$ to each multivariate system.
From the output of our computation, we obtain all elements $g$ in $G_1 = \{\, g \in \tilde{\gO}_{\varphi}(\overline{\mathbb{F}_{11}}) \mid g \cdot P_1^{\rm (alc)} \equiv r P_1^{\rm (alc)} \bmod{Q} \mbox{ for some } r \in {\overline{\mathbb{F}_{11}}}^{\times} \,\}$.
The set of the computed elements $g$ includes the matrices
\[
\begin{pmatrix}
6&3&5&3\\
1&6&1&5\\
5&3&6&3\\
1&5&1&6
\end{pmatrix},
\quad \text{and} \quad
\begin{pmatrix}
0&6&0&10\\
0&8&0&7\\
1&0&2&0\\
6&0&3&0
\end{pmatrix},
\]
whose orders are $2$ and $6$ respectively.
It is computationally checked that these two matrices generate $G_1 / {\sim}\cong \Aut (C_1^{\rm (alc)} )$.
\end{proof}

\section{Compatibility with the Galois cohomology theory}

In this section, we enumerate nonhyperelliptic superspecial curves over $\mathbb{F}_{11}$, using Galois cohomology theory.
However this enumeration requires the
data of isomorphism classes of nonhyperelliptic superspecial curves over
$\overline{\mathbb{F}_{11}}$ obtained in Corollary \ref{cor:KH17overF11} and Theorem \ref{theo:AlcAut}.
To be precise, what we show in this section is
that any $\F_{11}$-form of any curve in the list of Theorem \ref{thm:KH17overF11}
already appeared in the list.

Let $\Gamma$ denote the absolute Galois group $\mathrm{Gal}(\overline{\mathbb{F}_{11}} / \mathbb{F}_{11})$.
Let $C$ be a nonhyperelliptic superspecial curve over $\mathbb{F}_{11}$.
It is well-known (cf. \cite[Chap.III, \S 1.1, Prop. 1]{S}) that the first Galois cohomology $H^1 ( \Gamma, \mathrm{Aut} (C ) )$
parametrizes $\mathbb{F}_{11}$-forms of $C$,
where $\mathrm{Aut}(C)$ is the automorphism group over $\overline{\mathbb{F}_{11}}$.
Let $\sigma$ be the Frobenius on $\mathrm{Aut}(C)$.
The Galois cohomology is described as
\begin{equation}\label{GaloisCoh-Aut}
H^1 ( \Gamma, \mathrm{Aut} (C ) ) \cong \mathrm{Aut}(C ) / \sigma\text{-conjugacy} ,
\end{equation}
where two elements $a,b$ of $\mathrm{Aut}(C)$ are said to be
{\it $\sigma$-conjugate} if $a = g^{-1}bg^\sigma$ for some $g\in\mathrm{Aut}(C )$.
Let $a$ be an element of $\mathrm{Aut}(C )$.
Let $C^{(a)}$ be the $\F_{11}$-form associated to $a$ via \eqref{GaloisCoh-Aut}.
If $F$ is the Frobenius map on $C$, then the Frobenius on $C^{(a)}$
is given by $aF$ via an isomorphism
from $C\otimes \overline{\F_{11}}$ to $C^{(a)}\otimes \overline{\F_{11}}$.
Then $\Aut_{\F_{11}}\left(C^{(a)}\right)$ is bijective to the set
of $g\in\Aut(C)$ satisfying $g (aF) = (a F) g$. Hence we get
\[
\Aut_{\F_{11}}\left(C^{(a)}\right) \simeq \{g\in\mathrm{Aut}(C)\mid a = g^{-1}ag^\sigma\}.
\]
The RHS is just the $\sigma$-stabilizer group
$\sigma\text{-}{\rm Stab}_{\Aut(C)}(a)$ of $a$
in $\mathrm{Aut}(C)$, by definition.
Using Theorem \ref{theo:AlcAut} together with explicit representation by matrices (Proposition \ref{prop:ACAut}), the cardinality of $\mathrm{Aut}(C ) / \sigma\text{-conjugacy}$ for each $C_i$, and the orders of the $\sigma$-stabilizer groups of representatives of $\sigma$-conjugacy classes are computed as follows:
\begin{enumerate}
\item $|\mathrm{Aut}(C_1 ) / \sigma\text{-conjugacy}| = 6$
and the orders of the $\sigma$-stabilizer groups are $12, 6, 4, 6, 4, 12$;
\item  $|\mathrm{Aut}(C_2 ) / \sigma\text{-conjugacy}| = 2$
and the orders of the $\sigma$-stabilizer groups are $2,2$;
\item  $|\mathrm{Aut}(C_3 ) / \sigma\text{-conjugacy}| = 5$
and the orders of the $\sigma$-stabilizer groups are $24,3,8,4,4$;
\item  $|\mathrm{Aut}(C_4 ) / \sigma\text{-conjugacy}| = 6$
and the orders of the $\sigma$-stabilizer groups are $12, 4, 12, 4, 6, 6$;
\item  $|\mathrm{Aut}(C_5 ) / \sigma\text{-conjugacy}| = 5$
and the orders of the $\sigma$-stabilizer groups are $24, 4, 8, 3, 4$;
\item  $|\mathrm{Aut}(C_6 ) / \sigma\text{-conjugacy}| = 2$
and the orders of the $\sigma$-stabilizer groups are $2,2$;
\item  $|\mathrm{Aut}(C_7 ) / \sigma\text{-conjugacy}| = 1$
and the order of the $\sigma$-stabilizer group is $1$;
\item  $|\mathrm{Aut}(C_8 ) / \sigma\text{-conjugacy}| = 2$
and the orders of the $\sigma$-stabilizer groups are $2,2$;
\item  $|\mathrm{Aut}(C_9 ) / \sigma\text{-conjugacy}| = 1$
and the order of the $\sigma$-stabilizer group is $1$.
\end{enumerate}
Here we do not use Theorems \ref{thm:KH17overF11}.
One can see that the number of $\mathrm{Aut}(C_i) / \sigma\text{-conjugacy}$
coincides with the number of $\F_{11}$-forms of $C_i$
for each $i=1,2,\ldots, 9$ 
and that
the orders of the $\sigma$-stabilizer groups are exactly the same as those
of automorphism groups over $\F_{11}$ obtained in 
Propositions \ref{prop:N1Aut} -- \ref{prop:DegeAut}.
These support the correctness of our computational enumeration
proving Theorems \ref{thm:KH17overF11}.

For example, we first consider the case of $C_1$.
In this case we have $\Aut(C_1)\simeq \gD_{6}$
by Theorem \ref{theo:AlcAut} (1) and
any automorphism is defined over $\F_{11}$, i.e., $\sigma$ is trivial
by Proposition \ref{prop:ACAut} (1).
We know that the number of conjugacy classes of $\gD_{6}$ is $6$
and the orders of the stabilizer groups
are listed as above.
Next we consider the case of $C_8$. 
Recall $\Aut(C_8) \simeq \gA_4$.
By a straightforward computation,
as representatives of $\sigma$-conjugacy classes, we can take
the identity matrix and the second matrix in Proposition \ref{prop:ACAut} (8).
These two have the same $\sigma$-stabilizer group, which  is the cyclic group of order $2$ generated by
\[
\begin{pmatrix}
1 & 0 & 0& 0 \\
0 & 7 & 1& 7 \\
0 & 9 & 7& 1 \\
0 & 6 & 9& 7
\end{pmatrix}.
\]
The other cases are also studied in the same way.
For such explicit descriptions in all the cases,
see the web page \cite{KudoHP} of the first author, with codes producing these results.

\end{document}